\documentclass[11pt,a4paper]{article} 

\usepackage[utf8]{inputenc}    
\usepackage[T1]{fontenc}       
\usepackage{url}
\usepackage[nochapters]{classicthesis} 

\usepackage{framed}
\colorlet{shadecolor}{gray!10}

\usepackage{enumerate}

\usepackage{authblk}

\usepackage{xspace}

\RequirePackage[letterpaper, top=1in, bottom=1.5in, 
  left=1.5in, right=1.5in,showframe=false]{geometry}

\usepackage{mathtools}
\mathtoolsset{showonlyrefs}

\PassOptionsToPackage{%
    backend=biber, 
	language=auto,%
	style=numeric-comp,%
    bibstyle=authoryear, 
    sorting=nyt, 
    maxcitenames=3,
    mincitenames=1,
    maxbibnames=5, 
    minbibnames=1,
    natbib=true, 
    giveninits=true,
    bibstyle=ieee,
    doi=false,
    isbn=false,
    url=false,
}{biblatex}
\usepackage{biblatex}
\renewbibmacro*{bbx:savehash}{}
\DeclareFieldFormat{postnote}{#1}

\usepackage{amsmath}
\usepackage{amsthm}
\usepackage{array}

\usepackage{bbold} 

\newcommand{\E}[1]{{\mathbf E}\left[#1\right]}
\newcommand{\e}{{\mathbf E}}
\newcommand{\V}[1]{{\mathrm{Var}}\left(#1\right)}

\newcommand{\p}[1]{{\mathbf P}\left\{#1\right\}}

\newcommand\inprobHIGH{\,{\buildrel p \over \rightarrow}\,} 
\newcommand\inprob{{\inprobHIGH}}

\newcommand\inlawHIGH{\,{\buildrel d \over \rightarrow}\,} 
\newcommand\inlaw{{\inlawHIGH}}

\newcommand\inas{\,{\buildrel {a.s.} \over \rightarrow}\,} 

\newcommand{\eql}{\,{\buildrel \cL \over =}\,}

\newcommand{\eqd}{\,{\buildrel {\mathrm{def}} \over =}\,}

\newcommand{\argmin}{\operatornamewithlimits{argmin}}

\newcommand{\Bin}{\mathop{\mathrm{Bin}}}

\newcommand{\Poi}{\mathop{\mathrm{Poi}}}

\newcommand{\Unif}{\mathop{\mathrm{Unif}}}
\newcommand{\Exp}{\mathop{\mathrm{Exp}}}
\newcommand{\Gam}{\mathop{\mathrm{Gamma}}}

\newcommand{\iid}{i.i.d.\@\xspace}

\usepackage{mathtools}
\usepackage{stmaryrd}

\newcommand{\R}{{\mathbb R}}

\newcommand{\N}{{\mathbb N}}

\newcommand\cB{{\cal B}}

\newcommand\cF{{\cal F}}
\newcommand\cG{{\cal G}}

\newcommand\cK{{\cal K}}
\newcommand\cL{{\cal L}}

\newcommand\cN{{\cal N}}

\newcommand\cR{{\cal R}}
\newcommand\cS{{\cal S}}

\newcommand\cV{{\cal V}}

\newcommand\cX{{\cal X}}

\DeclarePairedDelimiter{\iverson}{\llbracket}{\rrbracket}

\DeclarePairedDelimiter{\ceil}{\lceil}{\rceil}

\newcommand\bigO[1]{{O\!\left( #1 \right)}}
\newcommand\smallo[1]{{o\!\left( #1 \right)}}

\newcommand\numberthis{\addtocounter{equation}{1}\tag{\theequation}}

\newcommand{\mySeq}[2]{(#1)_{#2}}

\newcommand{\HGFunc}{F}

\newcommand{\ee}{\mathrm{e}}

\newtheorem{lemma}{Lemma}

\newtheorem{theorem}{Theorem}

\newtheorem{proposition}{Proposition}

\theoremstyle{definition}

\newtheorem{remark}{Remark}
\newtheorem*{remark*}{Remark}

\newtheorem{example}{Example}

\newtheorem{myCond}{Condition}

\numberwithin{equation}{section}

\usepackage{textcomp}

\usepackage{eufrak}

\usepackage{hyperref} 
\hypersetup{colorlinks, breaklinks, urlcolor=Maroon, linkcolor=Maroon, citecolor=webgreen} 

\usepackage{tikz}
\usetikzlibrary{decorations.pathreplacing,angles,quotes}
\usetikzlibrary{shapes.geometric}
\usepackage[skip=2pt]{caption}
\usepackage{subcaption}

\newcommand{\kcut}{\cK}

\newcommand{\kcutE}{\cK_{e}}

\newcommand{\cutNum}{\cX}

\newcommand{\cutNumN}{\cutNum_n}

\newcommand{\clock}{T}
\newcommand{\clockRJ}[2]{\clock_{#1,#2}}

\newcommand{\clockiv}{\clockRJ{i}{v}}

\newcommand{\gclock}{G}
\newcommand{\gclockRJ}[2]{\gclock_{#1,#2}}

\newcommand{\gclockrv}{\gclockRJ{r}{v}}

\newcommand{\gclockOnej}{\gclockRJ{1}{j+1}}
\newcommand{\gclockOnei}{\gclockRJ{1}{i+1}}

\newcommand{\pclock}{T^{*}}
\newcommand{\pclockj}{\pclock_{1,j}}

\newcommand{\record}{I}
\newcommand{\recordRJ}[2]{\record_{#1,#2}}

\newcommand{\recordrv}{\recordRJ{r}{v}}
\newcommand{\recordOnej}{\recordRJ{1}{j}}

\newcommand{\precord}{I^{*}}
\newcommand{\precordnj}{\precord_{j}}

\newcommand{\recNum}{\cX}
\newcommand{\recNumN}{\recNum_n}
\newcommand{\recNumNi}[1]{\recNum_{n,#1}}
\newcommand{\recNumNr}{\recNumNi{r}}
\newcommand{\recNumNk}{\recNumNi{k}}
\newcommand{\recNumNOne}{\recNumNi{1}}

\newcommand{\precNum}{\cX^{*}}
\newcommand{\precNumN}{\precNum_{n}}

\newcommand{\grecord}{R}
\newcommand{\grecordnJ}[1]{\grecord_{n, #1}}
\newcommand{\grecordnj}{\grecordnJ{j}}
\newcommand{\grecordnp}{\grecordnJ{p}}
\newcommand{\grecordnSeq}{\mySeq{\grecordnp}{p \ge 1}}

\newcommand{\grecPos}{P}
\newcommand{\grecPosnJ}[1]{\grecPos_{n, #1}}
\newcommand{\grecPosnj}{\grecPosnJ{j}}
\newcommand{\grecPosnp}{\grecPosnJ{p}}
\newcommand{\grecPosnSeq}{\mySeq{\grecPosnp}{p \ge 1}}

\newcommand{\PRseq}{\mySeq{\grecordnp, \grecPosnp}{n \ge 1, p \ge 1}}

\newcommand{\gB}{B}

\newcommand{\gBnr}{\gB_{n,p}}
\newcommand{\gBnj}{\gB_{n,j}}
\newcommand{\gBns}{\gB_{n,s}}

\newcommand{\pB}{B^{*}}

\newcommand{\pBnr}{\pB_{n,p}}
\newcommand{\pBnj}{\pB_{n,j}}
\newcommand{\pBns}{\pB_{n,s}}

\newcommand{\erm}{{\mathrm e}}

\newcommand{\nscale}{{n^{1-\frac{1}{k}}}}

\newcommand{\Lpnr}{L^{*}_{n,p}}
\newcommand{\Lpnj}{L^{*}_{n,j}}

\newcommand{\Lnr}{L_{n,p}}

\newcommand{\SnOne}{S_{n,1}}
\newcommand{\Snr}{S_{n,p}}
\newcommand{\Sp}{S^{*}}
\newcommand{\Spnr}{S^{*}_{n,p}}
\newcommand{\Spnj}{S^{*}_{n,j}}
\newcommand{\SpnOne}{S^{*}_{n,1}}
\newcommand{\Snj}{S_{n,j}}

\newcommand{\Graphn}{\mathbb{G}_{n}}
\newcommand{\Treen}{\mathbb{T}_{n}}
\newcommand{\Pathn}{\mathbb{P}_{n}}
\newcommand{\Kn}{\mathbb{K}_{n}}

\begin{document}

\pagestyle{plain}
\title{\rmfamily\normalfont\spacedallcaps{\(k\)-cut on paths and some trees}\thanks{This work is supported by
    the Knut and Alice Wallenberg Foundation, the Swedish Research Council, and the Ragnar
    Söderbergs foundation. Emails:
    \texttt{\{xingshi.cai, cecilia.holmgren, fiona.skerman\}@math.uu.se},
    \texttt{lucdevroye@gmail.com}.}
\thanks{We thank Henning Sulzbach and Svante Janson for helpful discussions.}
}
\author{
    Xing Shi Cai, Cecilia Holmgren, Fiona Skerman\\
    \footnotesize{Uppsala University, Uppsala, Sweden}
    \and
    Luc Devroye\\
    \footnotesize{McGill University, Montreal, Canada}
}

\maketitle

\begin{abstract}
We define the (random) \(k\)-cut number of a rooted graph to model the difficulty of the
destruction of a resilient network.  The process is as the cut model of Meir and
Moon~\cite{meir70} except now a node must be cut \(k\) times before it is destroyed. The first
order terms of the expectation and variance of \(\recNumN\), the \(k\)-cut number of a path of
length \(n\), are proved.  We also show that \(\recNumN\), after rescaling, converges in
distribution to a limit \(\cB_{k}\), which has a complicated representation.  The paper then
briefly discusses the \(k\)-cut number of some trees and general graphs. We conclude by some analytic results which
may be of interest.
\end{abstract}
\section{Introduction and main results}
\label{sec:introduction}

\subsection{The \texorpdfstring{$k$}{k}-cut number of a graph}

Consider \(\Graphn\), a connected graph consisting of \(n\) nodes with exactly one node labeled as the
\emph{root}, which we call a \emph{rooted} graph.
Let \(k\) be a positive integer. We remove nodes from the graph as follows:
\begin{enumerate}
    \item Choose a node uniformly at random from the component that contains
        the root. Cut the selected node once.
    \item If this node has been cut \(k\) times, remove the node together with edges attached to it from the graph. 
    \item If the root has been removed, then stop.  Otherwise, go to step 1.
\end{enumerate}
We call the (random) total number of cuts needed to end this procedure the \(k\)-cut number and
denote it by \(\kcut(\Graphn)\). (Note that in traditional cutting models, nodes are removed as soon as
they are cut once, i.e., \(k=1\). But in our model, a node is only removed after being cut \(k\) times.)

One can also define an edge version of this process. Instead of cutting nodes, each time we choose
an edge uniformly at random from the component that contains the root and cut it once. If the edge
has been cut \(k\)-times then we remove it. The process stops when the root is isolated.  We let
\(\kcutE(\Graphn)\) denote the number of cuts needed for the process to end. 

Our model can also be applied to botnets, i.e., malicious computer networks consisting of
compromised machines which are often used in spamming or attacks. The nodes in \(\Graphn\) represent the
computers in a botnet, and the root represents the bot-master. The effectiveness of a botnet can be
measured using the size of the component containing the root, which indicates the resources
available to the bot-master \cite{4413000}. To take down a botnet means to reduce the size of this
root component as much as possible. If we assume that we target infected computers uniformly at
random and it takes at least \(k\) attempts to fix a computer, then the \(k\)-cut number measures
how difficult it is to completely isolate the bot-master.

The case \(k=1\) and \(\Graphn\) being a rooted tree has aroused great interests among
mathematicians in the past few decades.  The edge version of one-cut was first introduced by
\citet{meir70} for the uniform random Cayley tree.  \citeauthor{janson04} \cite{janson04, janson06}
noticed the equivalence between one-cuts and records in trees and studied them in binary trees and
conditional Galton-Watson trees.  Later \citet{ab14} gave a simpler proof for the limit distribution
of one-cuts in conditional Galton-Watson trees. For one-cuts in random recursive trees, see
\cite{meir74, iksanov07, drmota09}.  For binary search trees and split trees, see \cite{holmgren10,
holmgren11}.

\subsection{The \texorpdfstring{$k$}{k}-cut number of a tree}

One of the most interesting cases is when \(\Graphn = \Treen\), where \(\Treen\) is a rooted tree with
\(n\) nodes.

There is an equivalent way to define \(\kcut(\Treen)\). 
Imagine that each node is given an alarm clock. At time zero, the alarm clock of node \(v\) is set
to ring at time \(\clock_{1,v}\), where \( \mySeq{\clockiv}{i \ge 1, v \in \Treen}\) are \iid{}
(independent and identically distributed) \(\Exp(1)\) random variables. After the alarm clock of
node \(v\) rings the \(i\)-th time, we set it to ring again at time \(\clock_{i+1,v}\).
Due to the memoryless property of exponential random variables (see \cite[pp.~134]{durrett10}), at
any moment, which alarm clock rings next is always uniformly distributed. Thus, if
we cut a node that is still in the tree when its alarm clock rings, and remove the node with its
descendants if it has already been cut \(k\)-times, then we get exactly the \(k\)-cut model. 
(The random variables \(\mySeq{\clockRJ{i}{v}}{i \ge 1}\) can be seen as the holding times in a
Poisson process \(N(t)_{v}\) of parameter \(1\), where \(N(t)_{v}\) is the number of cuts in
\(v\) during the time \([0,t]\) and has a Poisson distribution with parameter \(t\).)

How can we tell if a node is still in the tree? When node \(v\)'s alarm clock rings for the \(r\)-th
time for some \(r \le k\), and no node above \(v\) has already rung \(k\) times, we say \(v\) has
become an \(r\)-\emph{record}. And when a node becomes an \(r\)-record, it must still be in the
tree. Thus, summing the number of \(r\)-records over \(r \in \{1,\dots,k\}\), we again get the
\(k\)-cut number \(\kcut(\Treen)\). One node can be a \(1\)-record, a \(2\)-record, etc., at the same
time, so it can be counted multiple times. Note that if a node is an \(r\)-record, then it must also
be a \(i\)-record for \(i \in \{1,\dots,r-1\}\).

To be more precise, we define \(\kcut(\Treen)\) as a function of \(\mySeq{\clockiv}{i \ge 1, v \ge 1}\).
Let 
\begin{align}
    \gclockrv \eqd \sum_{i=1}^{r} \clock_{i,v},
    \label{eq:gclockj}
\end{align}
i.e., \(\gclockrv\) is the moment when the alarm clock of node \(v\) rings for the \(r\)-th time.
Then \(\gclockrv\) has a gamma distribution with parameters \( (r,1)\) (see
\cite[Theorem~2.1.12]{durrett10}), which we denote by \(\Gam(r)\).
Let
\begin{align}
    \recordrv \eqd \iverson{\gclockRJ{r}{v} < \min\{\gclock_{k,u} : u \in \Treen, \text{\(u\) is an ancestor
        of \(v\)} \}},
    \label{eq:rec:def}
\end{align}
where 
\(\iverson{\cdot}\) denotes the Iverson bracket, i.e.,
\(\iverson{S}=1\) if the statement \(S\) is true and \(\iverson{S}=0\) otherwise.
In other words, \(\recordrv\) is the indicator random variable for node \(v\) being an \(r\)-record.
Let
\[
    \kcut_{r}(\Treen) \eqd \sum_{v \in \Treen} \recordrv,
    \qquad
    \kcut(\Treen) \eqd \sum_{r=1}^{k} \kcut_{r}(\Treen).
\]
Then \(\kcut_{r}(\Treen)\) is the number of \(r\)-records and \(\kcut(\Treen)\) is the total number of
records.

\subsection{The \texorpdfstring{$k$}{k}-cut number of a path}

Let \(\Pathn\) be a one-ary tree (a path) consisting of \(n\) nodes labeled \(1,\dots,n\) from the
root to the leaf. To simplify notations, from now on we use \(I_{r,i}, G_{r,i}\), and \(T_{r,i}\)
to represent \(I_{r,v}, G_{r,v}\) and \(T_{r,v}\) respectively for a node \(v\) at depth \(i\).

Let \(\cutNumN\eqd \kcut(\Pathn)\) and \(\recNumNr = \kcut_{r}(\Pathn)\).  In this
paper, we mainly consider \(\cutNumN\) and we let \(k \ge 2\) be a fixed integer.  

The first motivation of this choice is that, as shown in \autoref{sec:tree}, \(\Pathn\) is the fastest
to cut among all graphs. (We make this statement precise in Lemma~\ref{lem:universal}.) Thus
\(\cutNumN\) provides a universal stochastic lower bound for \(\kcut(\Graphn)\). Moreover, our results
on \(\cutNumN\) can immediately be extended to some trees of simple structures: see
Section~\ref{sec:tree}. Finally, as shown below, \(\cutNumN\) generalizes the well-known record
number in permutations and has very different behavior when \(k=1\), the usual cut-model, and \(k
    \ge 2\),  our extended model.

The name record comes from the classic definition of \emph{records} in random permutations. Let
\(\sigma_{1},\dots,\sigma_{n}\) be a uniform random permutation of \(\{1,\dots,n\}\). If
\(\sigma_{i} < \min_{1 \le j < i} \sigma_{j}\), then \(i\) is called a \emph{(strictly lower) record}.
Let \(\cR_{n}\) denote the number of records in \(\sigma_{1},\dots,\sigma_{n}\).  Let
\(W_{1},\dots,W_{n}\) be \iid{} random variables with a common continuous distribution.  Since the
relative order of \(W_{1},\dots,W_{n}\) also gives a uniform random permutation, we can equivalently
define \(\sigma_{i}\) as the rank of \(W_{i}\). As gamma distributions are continuous, we can in
fact let \(W_{i} = \gclockRJ{k}{i}\).  Thus, being a record in a uniform permutation is equivalent to
being a \(k\)-record and \(\cR_{n} \eql \recNumNi{k}\). Moreover, when \(k=1\), \(\cR_{n} \eql
\recNumN\).

Starting from \citeauthor{chandler52}'s article \cite{chandler52} in \citeyear{chandler52}, the
theory of records has been widely studied due to its applications in statistics, computer science,
and physics. For more recent surveys on this topic, see \cite{ahsanullah2004record}.

A well-known result of \(\cR_n\) (and thus also \(\recNumNi{k}\)) \cite{renyi62} is that \(
\mySeq{\recordRJ{k}{j}}{1 \le j \le n}\) are independent. It follows from 
the Lindeberg–Lévy–Feller Theorem that
\begin{align}
    \frac{\E{\cR_{n}}}{\log n}
    \to
    1,
    \qquad
    \frac{\cR_{n}}{\log n}
    \inas
    1
    ,
    \qquad
    \cL
    \left( 
    \frac{\cR_{n} - \log n}{\sqrt{\log n}}
    \right)
    \inlaw
    \cN(0,1)
    ,
    \label{eq:clt:record}
\end{align}
where \(\cN(0,1)\) denotes the standard normal distribution.

In the following, \autoref{thm:rec:mean} gives the expectation of \(\recNumNr\) which implies that
the number of one-records dominates the number of other records. Subsequently
\autoref{thm:rec:var} and \autoref{thm:rec:high} estimate the variance and higher moments of
\(\recNumNOne\).
\begin{theorem}
    \label{thm:rec:mean}
    For all fixed \(k \in \N\),
    \begin{align}
        \E{\recNumNr}  
        \sim 
        \left\{
        \begin{array}{*2{>{\displaystyle}l}}
        \eta_{k, r} n^{1-\frac{r}{k}}
        \qquad
        \qquad
        & 
        (1 \le r < k),
        \\
        \log n
        &
        (r=k)
        ,
        \end{array}
        \right.
        \label{eq:rec:mean:k}
    \end{align}
    where the constants \(\eta_{k, r}\) are defined by
    \begin{equation}
        \eta_{k, r} 
        \eqd 
        \frac{(k!)^{\frac{r}{k}}}{k-r} \frac{\Gamma\left( \frac{r}{k} \right)}{\Gamma(r)}
        ,
        \label{eq:eta}
    \end{equation}
    where \(\Gamma(z)\) denotes the gamma function.
    Therefore \(\E{\recNumN} \sim \E{\recNumNOne}\).
    Also, for \(k=2\),
    \[
        \E{\recNumN}  
        \sim
        \E{\recNumNOne}  
        \sim
        \sqrt{2 \pi n}
        .
    \]
\end{theorem}

\begin{theorem}
    \label{thm:rec:var}
    For all fixed \(k \in \{2,3,\dots\}\),
    \begin{align}
        \E{\recNumNOne(\recNumNOne - 1)}
        \sim
        \E{\left(\recNumNOne\right)^{2}}
        \sim
        \gamma_{k}
        n^{2-\frac{2}{k}}
        ,
        \label{eq:rec:factorial}
    \end{align}
    where
    \begin{align}
        \gamma_{k} =
        \frac{\Gamma\left( \frac{2}{k} \right) (k!)^{\frac{2}{k}}}{k-1}
        +
        2 \lambda_{k}
        ,
        \label{eq:gamma}  
    \end{align}
    and
    \begin{align}
        \label{eq:xi:k:2}
        \lambda_{k} =
        \left\{
        \begin{array}{*2{>{\displaystyle}l}}
        \frac{\pi  \cot \left(\frac{\pi }{k}\right) \Gamma \left(\frac{2}{k}\right)
            (k!)^{\frac{2}{k}}}{2 \left({k}-2\right) \left({k}-1\right)}
        & 
        k>2,
        \\
        \frac{\pi^2}{4}
        &
        k=2
        .
        \end{array}
        \right.
    \end{align}
    Therefore
    \begin{align}
        \V{\recNumNOne}
        \sim
        \left(
        \gamma_{k}
        -
        \eta_{k,1}^{2}
        \right)
        n^{2-\frac{2}{k}}
        .
        \label{eq:rec:var:k}
    \end{align}
    In particular, when \(k = 2\)
    \begin{align}
        \V{\recNumNOne}
        \sim
        \left(
        \frac{\pi^2}{2} + 2
        -
        2 \pi
        \right)
        n
        .
        \label{eq:rec:var:k2}
    \end{align}
\end{theorem}

\begin{theorem}\label{thm:rec:high}
    For all fixed \(k \in \{2,3,\dots\}\) and \(\ell \in \N\)
    \begin{equation}\label{eq:rec:high}
        \limsup_{n \to \infty} \E{\left(\frac{\recNumNOne}{n^{1-\frac{1}{k}}} \right)^{\ell}} 
        \le
        \rho_{k,\ell}
        \eqd
        {\ell! }
        {\Gamma \left(\ell+1-\frac{\ell}{k}\right)^{-1}}
        \left(\frac{\pi}{k}(k!)^{1/k} \sin\left(\frac{\pi }{k}\right)^{-1} \right)^\ell
        .
    \end{equation}
    The upper bound is tight for \(\ell=1\) since \(\rho_{k,1}=\eta_{k,1}\).
\end{theorem}

The above theorems imply that the correct rescaling parameter should be \(\nscale\).  However,
unlike the case \(k=1\), when \(k\ge2\) the limit distribution of \(\recNumN/\nscale\) has a rather
complicated representation \(\cB_{k}\) defined as follows:  Let
\(U_{1},E_{1},U_{2},E_{2},\dots\) be mutually independent random variables with \(E_{j} \eql
    \Exp(1)\) and \(U_{j} \eql \Unif[0,1]\).  Let
\begin{align}
    &
    S_{p}
    \eqd
    \left( 
        k! 
        \sum_{1 \le s \le p}
        \left( 
            \prod_{s \le j < p} U_{j}
        \right)
        E_{s}
    \right)^{\frac{1}{k}}
    ,
    \label{eq:Sr}
    \\
    &
    B_{p}
    \eqd
    \left( 1-U_{p} \right)
    \left( \prod_{1 \le j < p} U_{j} \right)^{1-\frac{1}{k}}
    S_{p}
    ,
    \label{eq:Br}
    \\
    &
    \cB_{k}
    \eqd
    \sum_{1 \le p} B_{p}
    ,
    \label{eq:cBk}
\end{align}
where we use the convention that an empty product equals one.

\begin{remark}
    \label{rm:Sr}
    An equivalent recursive definition of \(S_{p}\) is
    \begin{align}
        S_{p}
        =
        \begin{cases}
            k! E_{1} & (p=1), \\
            \left( U_{p-1} S_{p-1}^{k} + k! E_{p} \right)^{\frac{1}{k}} \qquad\qquad & (p \ge 2).
        \end{cases}
        \label{eq:Sr:rec}
    \end{align}
\end{remark}

\begin{theorem}
    \label{thm:rec:limit}
    Let \(k \in \{2,3,\dots\}\).
    Let \(\cL(\cB_{k})\) denote the distribution of \(\cB_{k}\).
    Then \begin{align}
        \cL
        \left( 
        \frac{\recNumN}{\nscale}
        \right)
        \inlaw
        \cL(\cB_{k})
        .
        \label{eq:rec:limit}
    \end{align}
    Thus, by \autoref{thm:rec:mean}, \ref{thm:rec:var} and \ref{thm:rec:high},
    the convergence also holds in \(L^{p}\) for all \(p>0\) and
    \begin{equation}\label{eq:Bk:momment}
        \E{ \cB_{k} } = \eta_{k,1}
        ,
        \qquad
        \E{\cB_{k}^{2}}
        =
        \gamma_{k}
        ,
        \qquad
        \E{\cB_{k}^{p}}
        \in [\eta_{k,1}^{p}, \rho_{k,p}]
        \qquad
        (p \in \N)
        .
    \end{equation}
\end{theorem}

\begin{remark}
    The idea behind \(\cB_{k}\) is that we split the path into segments according to the positions
    of \(k\)-records, then count the numbers of one-records in every segment, each of which
    converges to a \(B_{p}\) in the sum \eqref{eq:cBk}.  This will be made rigorous in
    \autoref{sec:limit}. We will also see that \(\cB_{k}\) has a density
    close to a normal distribution in \autoref{sec:B:k}.
\end{remark}

\begin{remark}
It is easy to see that \(\cutNum_{n+1}^{\mathrm e} \eqd \kcutE(P_{n+1}) \eql \cutNumN\) by treating
each edge on a length \(n+1\) path as a node on a length \(n\) path.
\end{remark}

The rest of the paper is organized as follows: \autoref{sec:moment} proves the moment results
\autoref{thm:rec:mean}, \ref{thm:rec:var}, and \ref{thm:rec:high}.  \autoref{sec:limit}
deals with the distributional result \autoref{thm:rec:limit}.  \autoref{sec:tree} discusses some easy
results for general graphs and trees.  Finally, \autoref{sec:int} collects 
analytic results used in the proofs, which may themselves be of interest.

\section{The moments}
\label{sec:moment}

\subsection{The expectation}
\label{sec:mean}

In this section we prove \autoref{thm:rec:mean}.

\begin{lemma}
    \label{lem:mean:i}
    Uniformly for all \(i \ge 1\) and \(r \in \{1,\dots,k\}\),
    \begin{equation}
        \E{\recordRJ{r}{i+1}}
        =
        \left( 1 + \bigO{i^{-\frac{1}{2k}}} \right)
        \frac{(k!)^{\frac{r}{k}} }{k}
        \frac{\Gamma\left( \frac{r}{k} \right)}{\Gamma(r)}
        i^{-\frac{r}{k}}
        .
        \label{eq:mean:i}
    \end{equation}
\end{lemma}

\begin{proof}
    By \eqref{eq:rec:def},
    \(\E{\recordRJ{r}{i+1}}=\p{\gclockRJ{k}{1}>\gclockRJ{r}{i+1},\dots,\gclockRJ{k}{i}>\gclockRJ{r}{i+1}}\).
    Conditioning on \(\gclockRJ{r}{i+1}=x\) yields \(\E{\recordRJ{r}{i+1}} = \int_{0}^{\infty}
        {x^{r-1} \erm^{-x}}/{\Gamma(r)} \p{\gclockRJ{k}{1} > x}^{i} \, \mathrm{d}x\).
    \autoref{lem:mean:i} thus follows from
    \autoref{lem:i:int}.
\end{proof}

\begin{proof}[Proof of Theorem {\ref{thm:rec:mean}}]
    A simply computation shows that for \(a \in (0,1)\) 
    \begin{equation}
        \sum_{1 \le i \le n} \frac{1}{i^{a}}
        =
        \frac{1}{1-a} n^{1-a}
        +
        O(1)
        .
        \label{eq:em:ia}
    \end{equation}
    It then follows from \autoref{lem:mean:i} that for \(r \in \{1,\dots,k-1\}\).
    \begin{equation}
        \label{eq:mean:r}
        \E{\recNumNr}  
        = 
        \sum_{0 \le i < n}
        \E{\recordRJ{r}{i+1}}
        =
        \frac{(k!)^{\frac{r}{k}} }{k}
        \frac {\Gamma\left( \frac{r}{k} \right)} {\Gamma(r)}
        \frac{1}{1-\frac{r}{k}}
        n^{1-\frac{r}{k}}
        +
        \bigO{n^{1-\frac{r}{k}-\frac{1}{2k}}}
        +
        \bigO{1}
        .
    \end{equation}
    When \(r=k\), \(\E{\recNumNi{k}} = \E{\cR_{n}} \sim \log(n) \) is already well-known.
\end{proof}

\subsection{The variance}
\label{sec:var}

In this section we prove \autoref{thm:rec:var}. 

Let \(E_{i,j}\) denote the event that \([\recordRJ{1}{i+1}\recordRJ{1}{j+1}=1]\).
Let \(A_{x,y}\) denote the event that \([\gclockOnei = x \cap \gclockOnej = y]\).
Then conditioning on \(A_{x,y}\)
\begin{equation}
    E_{i,j}
    =
    \left[ 
        \bigcap_{1 \le s \le i}
        \gclockRJ{k}{s} > x \vee y
    \right]
    \cap
    \left[ 
        \gclockRJ{k}{i+1} > y
    \right]
    \cap
    \left[ 
        \bigcap_{i+2 \le s \le j}
        \gclockRJ{k}{s} > y
    \right]
    ,
    \label{eq:i:j:event}
\end{equation}
where \(x \vee y \eqd \max\{x,y\}\).
Since conditioning on \(A_{x,y}\), \(\gclockRJ{k}{i+1} \eql \Gam(k-1) + x\), \(\gclockRJ{k}{s} \eql \Gam(k)\) for \(s \notin
\{i+1,j+1\}\), and all these random variables are independent, we have
\begin{align}
    \p{E_{i,j}|A_{x,y}}
    =
    \p{G_{k-1,1} + x > y}
    \p{G_{k,1} > x \vee y}^{i}
    \p{G_{k,1} > y}^{j-i-1}
    .
    \label{eq:i:j:prob}
\end{align}
It follows from \(\gclockOnei \eql \gclockOnej \eql \Exp(1)\) that
\begin{align}
    \p{E_{i,j}}
    &
    =
    \int_{0}^{\infty}
    \int_{y}^{\infty}
    \erm^{-x-y}
    \p{E_{i,j}|A_{x,y}}
    \, \mathrm{d}x
    \, \mathrm{d}y
    \\
    &
    \qquad
    +
    \int_{0}^{\infty}
    \int_{0}^{y}
    \erm^{-x-y}
    \p{E_{i,j}|A_{x,y}}
    \, \mathrm{d}x
    \, \mathrm{d}y
    \\
    &
    \eqd
    A_{1,i,j}+ A_{2,i,j}
    .
    \label{eq:int:var:ij}
\end{align}
We next estimate these two terms.

\begin{lemma}
    \label{lem:A2}
    Let \(k \in \{2,3,\dots\}\).  We have
    \begin{align}
        A_{2,i,j}
        =
        \left( 1+\bigO{j^{-\frac{1}{2k}}} \right)
        \frac{(k!)^{\frac{2}{k}} }{k}
        {\Gamma\left( \frac{2}{k} \right)}
        j^{-\frac{2}{k}}
        .
        \label{eq:A2:1}
    \end{align}
\end{lemma}

\begin{proof}
    In this case, \(x \vee y=y\). Thus, by \eqref{eq:i:j:prob}
\begin{align}
    \label{eq:A2}
    A_{2,i,j}
    =
    \int_{0}^{\infty}
    \erm^{-y}
    \p{G_{k,1} > y}^{j-1}
    \int_{0}^{y}
    \erm^{-x}
    \p{G_{k-1,1}  > y-x}
    \, \mathrm{d}x
    \, \mathrm{d}y
    .
\end{align}
Note that the dependence on \(i\) disappears.
Let \(Z\) denote a Poisson random variable with mean \(y-x\).  By the well-known connection between
Poisson and gamma distributions, the inner integral in the above equals
\begin{align}
    \int_{0}^{y}
    \erm^{-x}
    \p{Z<k-1}
    \, \mathrm{d}x
    =
    \int_{0}^{y}
    \erm^{-x}
    \sum_{\ell = 0}^{k-2}
    \erm^{-(y-x)}
    \frac
    {(y-x)^{\ell}}
    {\ell !}
    \, \mathrm{d}x
    =
    \erm^{-y}
    \sum_{\ell = 0}^{k-2}
    \frac
    {y^{\ell+1}}
    {(\ell+1)!}
    .
    \label{eq:A2:inner}
\end{align}
It then follows from \autoref{lem:i:int} that
\begin{align}
    \label{eq:A2:final}
    A_{2,i,j}
    &
    =
    \sum_{\ell = 0}^{k-2}
    \int_{0}^{\infty}
    \erm^{-2y}
    \frac
    {y^{\ell+1}}
    {(\ell+1)!}
    \p{G_{k,1} > y}^{j-1}
    \, \mathrm{d}y
    \\
    &
    =
    \sum_{\ell = 0}^{k-2}
    \left( 1+\bigO{j^{-\frac{1}{2k}}} \right)
    \frac{\left( k! \right)^{\frac{\ell+2}{k}}}{k(\ell+1)!}
    \Gamma \left(\frac{\ell+2}{k} \right)
    j^{-\frac{\ell+2}{k}}
    \\
    &
    =
    \left( 1+\bigO{j^{-\frac{1}{2k}}} \right)
    \frac{(k!)^{\frac{2}{k}} }{k}
    {\Gamma\left( \frac{2}{k} \right)}
    j^{-\frac{2}{k}}
    .
    \qedhere
\end{align}
\end{proof}

\begin{lemma}
    \label{lem:A1}
    Let \(k \in \{2,3,\dots\}\).
    Let \(a=i\) and \(b=j-i-1\). 
    Then for all \(a \ge 1\) and \(b \ge 1\),
    \begin{align}
        A_{1,i,j}
        =
        \xi_{k}(a,b)
        +
        \bigO{
            \left( 
                a^{-\frac{1}{2k}}
                +
                b^{-\frac{1}{2k}}
            \right)
            \left( 
                a^{-\frac{2}{k}}
                +
                b^{-\frac{2}{k}}
            \right)
        }
        ,
        \label{eq:A1}
    \end{align}
    where
    \begin{align}
        \label{eq:xi:def}
        \xi_{k}(a,b)
        \eqd
        \int_{0}^{\infty}
        \int_{y}^{\infty}
        \exp
        \left( 
            -a \frac{x^{k}}{k!}
            -b \frac{y^{k}}{k!}
        \right)
        \, \mathrm{d}x
        \, \mathrm{d}y
        .
    \end{align}
\end{lemma}

\begin{proof}
    In this case, \(x\vee y=x\) and \(y-x<0\). 
    Thus, by \eqref{eq:i:j:prob} and \autoref{lem:i:int}
\begin{align}
    A_{1,i,j}
    &
    =
    \int_{0}^{\infty}
    \int_{y}^{\infty}
    \erm^{-x} \erm^{-y}
    \p{G_{k,1}>x}^{i}
    \p{G_{k,1}>y}^{j-i-1}
    \, \mathrm{d}x
    \, \mathrm{d}y
    \\
    &
    =
    \int_{0}^{\infty}
    \int_{y}^{\infty}
    \erm^{-x-y}
    \left( 
        \frac{\Gamma(k,x)}{\Gamma\left( k \right)}
    \right)^{a}
    \left( 
        \frac{\Gamma(k,y)}{\Gamma\left( k \right)}
    \right)^{b}
    \,
    \mathrm{d} x
    \,
    \mathrm{d} y
    ,
    \label{eq:var:2}
\end{align}
where \(\Gamma(\ell,z)\) denotes the upper incomplete gamma function.

Let \(\cS\) be the integration area of \eqref{eq:var:2}. 
Let \(x_{0}=a^{-\alpha}\) and \(y_{0}=b^{-\alpha}\) where \(\alpha=\frac{1}{2}\left(
\frac{1}{k}+\frac{1}{k+1} \right)\). 
Let
\begin{align}
    \cS_{0} = \cS \cap \left\{(x,y) \in \R^{2}: x < x_0, y < y_0 \right\}.
    \label{eq:var:3}
\end{align}
We split \eqref{eq:var:2} into two parts \(A_{1,1}\) and \(A_{1,2}\) with integration area
\(\cS_{0}\) and \(\cS\setminus\cS_{0}\) respectively.

Let \(\beta=\frac{1}{2(k+1)}\).
Let \(x_{1} = a^{\beta}/k!\) and \(y_{1} = b^{\beta}/k!\).
It follows from \autoref{lem:gam:approx} and \autoref{lem:bound:xi} that
\begin{align}
    A_{1,1}
    &
    =
    \left( 
        1
        +
        \bigO
        {
            a^{-\frac{1}{2k}}
            +
            b^{-\frac{1}{2k}}
        }
    \right)
    \iint_{\cS_{0}}
    \exp
    \left( 
        -a \frac{x^{k}}{k!}
        -b \frac{y^{k}}{k!}
    \right)
    \, \mathrm{d}x
    \, \mathrm{d}y
    \\
    &
    =
    \left( 
        1
        +
        \bigO
        {
            a^{-\frac{1}{2k}}
            +
            b^{-\frac{1}{2k}}
        }
    \right)
    \xi_{k}(a,b)
    +
    \bigO{
        \erm^{-x_{1}}
        +
        \erm^{-y_{1}}
    }
    \\
    &
    =
    \xi_{k}(a,b)
    +
    \bigO{
        \left( 
            a^{-\frac{1}{2k}}
            +
            b^{-\frac{1}{2k}}
        \right)
        \left( 
            a^{-\frac{2}{k}}
            +
            b^{-\frac{2}{k}}
        \right)
    }
    +
    \bigO{
        \erm^{-x_{1}}
        +
        \erm^{-y_{1}}
    }
    .
    \label{eq:var:5}
\end{align}
It is not difficult to verify that
\begin{align}
    A_{1,2}
    =
    \bigO{
        \left( 
            \frac{
                \Gamma(k,x_{0})
            }{\Gamma(k)}
        \right)^{-a}
        +
        \left( 
            \frac{
                \Gamma(k,y_{0})
            }{\Gamma(k)}
        \right)^{-b}
    }
    =
    \bigO{
        \erm^{-x_{1}}
        +
        \erm^{-y_{1}}
    }
    .
    \label{eq:var:9}
\end{align}
The lemma follows since \( \erm^{-x_{1}} + \erm^{-y_{1}}\) is exponentially
small.
\end{proof}

\begin{proof}[Proof of \autoref{thm:rec:var}]
We have
\begin{align}
    \E{\recNumNOne\left(\recNumNOne-1 \right)}
    =
    2
    \sum_{i=1}^{n-1}
    \sum_{j=i+1}^{n}
    \p{E_{i,j}}
    =
    2 
    \sum_{i=0}^{n-2}
    \sum_{j=i+1}^{n-1}
    \left( 
        A_{1,i,j}
        +
        A_{2,i,j}
    \right)
    .
    \label{eq:var:11}
\end{align}
Thus, by \autoref{lem:A2} and \eqref{eq:em:ia},
\begin{align}
    \sum_{i=0}^{n-2}
    \sum_{j=i+1}^{n-1}
    A_{2,i,j}
    &
    =
    \sum_{i=0}^{n-2}
    \sum_{j=i+1}^{n-1}
    \left[ 
    \frac{(k!)^{\frac{2}{k}} {\Gamma\left( \frac{2}{k} \right)}}{k}
    j^{-\frac{2}{k}}
    +
    \bigO{j^{-\frac{5}{2k}}}
    \right]
    \\
    &
    =
    \frac{(k!)^{\frac{2}{k}} \Gamma\left( \frac{2}{k} \right)}{2(k-1)}n^{2-\frac{2}{k}}
    +
    \bigO{n^{2-\frac{5}{2k}}}
    .
    \label{eq:var:13}
\end{align}
For \(A_{1,i,j}\), it follows from \autoref{lem:A1} that
\begin{align}
    \sum_{i=0}^{n-2}
    \sum_{j=i+1}^{n-1}
    A_{1,i,j}
    &
    =
    \sum_{a=1}^{n-1}
    \,
    \sum_{b=1}^{n-a}
    \xi_{k}(a,b)
    +
    \bigO{n^{2-\frac{5}{2k}}}
    \\
    &
    =
    \int_{0}^{n}
    \,
    \int_{0}^{n-a}
    \xi_{k}(a,b)
    \,
    \mathrm{d}b
    \,
    \mathrm{d}a
    +
    \bigO{n^{2-\frac{5}{2k}}}
    \\
    &
    =
    n^{2-\frac{2}{k}}
    \int_{0}^{1}
    \,
    \int_{0}^{1-s}
    \xi_{k}(s,t)
    \,
    \mathrm{d}t
    \,
    \mathrm{d}s
    +
    \bigO{n^{2-\frac{5}{2k}}}
    \\
    &
    =
    \lambda_{k}
    n^{2-\frac{2}{k}}
    +
    \bigO{n^{2-\frac{5}{2k}}}
    ,
    \label{eq:var:14}
\end{align}
where the last step follows from \autoref{lem:int:hyper}.
\autoref{thm:rec:var} follows by putting \eqref{eq:var:13}, \eqref{eq:var:14} into
\eqref{eq:var:11}.
\end{proof}

\subsection{Higher moments}
\label{sec:high}

In this section we prove \autoref{thm:rec:high}.

The computations of higher moments of \(\recNumNOne\) are rather complicated. However, an upper bound is readily available. Let \((x)_{\ell}\eqd
    x(x-1)\dots(x-\ell+1)\).  For \(\ell \ge 1\),
\begin{align*}
    \E{(\recNumNOne)_{\ell}}
    &
    =
    \ell!
    \sum_{1 \le i_{1} < i_{2} \dots < i_{\ell} \le n}
    \E{\recordRJ{1}{i_{1}} \recordRJ{1}{i_{2}} \cdots \recordRJ{1}{i_{\ell}}}
    \\
    &
    \le
    \ell!
    \sum_{1 \le i_{1} < i_{2} \dots < i_{\ell} \le n}
    \E{\recordRJ{1}{i_{1}}}
    \E{\recordRJ{1}{i_{2}-i_{1}}}
    \cdots
    \E{\recordRJ{1}{i_{\ell}-i_{\ell-1}}}
    \\
    &
    =
    \ell!
    \sum_{(a_{1}, \dots,a_{\ell}) \in \cS_{n,\ell}}
    \prod_{j=1}^{\ell}
    \E{\recordRJ{1}{a_{j}}}
    ,
    \numberthis\label{eq:high:1}
\end{align*}
where
\begin{equation}\label{eq:Snp}
    \cS_{n,\ell}
    \eqd
    \left\{
        (a_{1},a_{2},\dots,a_{\ell}) \in \N^{\ell}:
        a_{1} \ge 0, \dots, a_{\ell} \ge 0, \sum_{j=1}^{\ell} a_{j} \le n-\ell
    \right\}
        .
\end{equation}
The above inequality holds since if \(i_{j}\) is a one-record in the whole path, then it must also
be a one-record in the segment \( (i_{j-1}+1,\dots,i_{j})\) ignoring everything else, and what
happens in each of such segments are independent. It follows from \autoref{lem:mean:i} that
\eqref{eq:high:1} equals
\begin{equation}\label{eq:high:2}
    \begin{aligned}
        &
        \ell!
    \sum_{(a_{1}, \dots,a_{\ell}) \in \cS_{n,\ell}}
    \prod_{j=1}^{\ell}
    \left( 1 + \bigO{a_{j}^{-\frac{1}{2k}}} \right)
    \frac{(k!)^{\frac{1}{k}} }{k}\Gamma\left( \frac{1}{k} \right)
    a_{j}^{-\frac{1}{k}}
    \\
    &
    =
        \ell!
    n^{\ell\left( 1-\frac{1}{k} \right)}
    \left( 
    \frac{(k!)^{\frac{1}{k}} }{k}\Gamma\left( \frac{1}{k} \right)
    \right)^{\ell}
    \sum_{(a_{1}, \dots,a_{\ell}) \in \cS_{n,\ell}}
    \prod_{j=1}^{\ell}
    \left( 1 + \bigO{a_{j}^{-\frac{1}{2k}}} \right)
    \left(\frac{a_{j}}{n} \right)^{-\frac{1}{k}}
    \frac{1}{n}
    \\
    &
    \sim
    n^{\ell\left( 1-\frac{1}{k} \right)}
        \ell!
    \left( 
    \frac{(k!)^{\frac{1}{k}} }{k}\Gamma\left( \frac{1}{k} \right)
    \right)^{\ell}
    \int_{A_{\ell}}
    \prod_{j=1}^{\ell} x_{j}^{-\frac{1}{k}}
    \, \mathrm{d}(x_{1},\dots,x_{\ell})
    \\
    &
    =
    n^{\ell\left( 1-\frac{1}{k} \right)}
    \ell!
    \left( 
    \frac{(k!)^{\frac{1}{k}} }{k}\Gamma\left( \frac{1}{k} \right)
    \right)^{\ell}
    \zeta_{k,\ell}
    =
    n^{\ell\left( 1-\frac{1}{k} \right)}
    \rho_{k,\ell}
    ,
    \end{aligned}
\end{equation}
where \(A_{\ell}\) is the simplex
\(
    \{(x_{1},\dots,x_{\ell}):x_{1}>0,\dots,x_{\ell}>0,x_{1}+\dots+x_{\ell}<1\},
\)
and
\begin{equation}\label{eq:int:k}
    \zeta_{k,\ell}
    \eqd
    {\Gamma \left(\frac{k-1}{k}\right)^\ell}{\Gamma
        \left(1+\ell-\frac{\ell}{k}\right)^{-1}}
    .
\end{equation}
The above integral is known as the \emph{beta} integral \cite[5.14.1]{NIST}.

\section{Convergence to the \texorpdfstring{\(k\)}{k}-cut distribution}\label{sec:limit}

By \autoref{thm:rec:mean} and Markov's inequality, \({\recNumNr}/{\nscale} \inprob 0\) for \(r \in
    \{2,\dots,k\}\).  So it suffices to prove Theorem \ref{thm:rec:limit}  for \(\recNumNOne\)
instead of \(\recNumN\). Throughout \autoref{sec:limit}, unless otherwise emphasized, we assume that
\(k \ge 2\).

The idea of the proof is to condition on the positions and values of the \(k\)-records, and study the
distribution of the number of one-records between two consecutive \(k\)-records.

We use \(\grecordnSeq\) to denote the \(k\)-record values and
\(\grecPosnSeq\) the positions of these \(k\)-records.
To be precise,  let \(\grecordnJ{0} \eqd 0\), and \(\grecPosnJ{0} \eqd n+1\); for \(p \ge 1\), 
if \(\grecPosnJ{p-1}>1\), then let
\begin{equation}
    \begin{aligned}
    & 
    \grecordnJ{p} \eqd \min \{ \gclockRJ{k}{j}:{1 \le j < \grecPosnJ{p-1}} \},
    \\
    &
    \grecPosnJ{p} \eqd \argmin \{ \gclockRJ{k}{j}:{1 \le j < \grecPosnJ{p-1}} \} ,
    \end{aligned}
    \label{eq:grecord}
\end{equation}
i.e., \(\grecPosnJ{p}\) is the unique positive integer which satisfies that
\(\gclockRJ{k}{\grecPosnJ{p}} \le \gclockRJ{k}{i}\) for all \(1 \le i <
    \grecPosnJ{p-1}\);
otherwise let \(\grecPosnJ{p}=1\) and \(\grecordnJ{p}=\infty\). Note that \(\grecord_{n,1}\) is
simply the minimum of \(n\) \iid{} \(\Gam(k)\) random variables.

According to \(\grecPosnSeq\), we split \(\recNumNOne\) into the following sum
\begin{equation}
    \recNumNOne
    = \sum_{1 \le j \le n} \recordOnej
    = 
    \recNumNk
    +
    \sum_{1 \le p} 
    \sum_{1 \le j}
    \,
    \iverson{\grecPosnJ{p-1} > j > \grecPosnJ{p}}
    \,
    \recordOnej
    \eqd
    \recNumNk
    +
    \sum_{1 \le p} \gBnr
    .
    \label{eq:gBnr}
\end{equation}
Figure \ref{fig:seg} gives an example of \(\mySeq{\gBnr}{p \ge 1}\) for \(n = 12\).
It depicts the positions of the \(k\)-records and the one-records. It also shows the values and the
summation ranges for \(\mySeq{\gBnr}{p \ge 1}\).

\begin{figure}[ht]
\begin{center}
\begin{tikzpicture}[scale=0.95]
    \tikzstyle{k record}=[circle,inner sep=5pt,draw=blue!50,fill=blue!20]
    \tikzstyle{one record}=[rectangle,inner sep=4pt,draw=black,fill=red!0]
    \tikzstyle{no record}=[circle,inner sep=2pt,draw=black,fill=black]
    \tikzstyle{Bnr bracket}=[decoration={brace,mirror,raise=0pt},decorate]
    \pgfmathsetmacro{\YL}{-0.8}
    \pgfmathsetmacro{\YLL}{-1.8}
    \pgfmathsetmacro{\YLLL}{0.8}
    \draw (0,0) -- (12,0);
    \draw[dashed] (12,0) -- (13,0);
    \foreach \x in {0,3,6}
    {
    \node[k record] at (\x,0) {};
    }
    \foreach \x in {0,2,3,4,6,9,11}
    {
    \node[one record] at (\x,0) {};
    }
    \foreach \x in {0,...,12}
    {
        \node[no record] at (\x,0) {};
    }
    \foreach \x in {0,...,13}
    {
        \node[anchor=north] at (\x,-0.3) {$\x$};
    }
    \foreach \x in {0.5, 2.5, 3.5, 5.5, 6.5, 12.5}
    {
        \draw[dotted] (\x, \YLLL) -- (\x, 0);
    }
    \node[anchor=north] at (0,\YL) {$\grecPos_{n,3}$};
    \node[anchor=north] at (3,\YL) {$\grecPos_{n,2}$};
    \node[anchor=north] at (6,\YL) {$\grecPos_{n,1}$};
    \node[anchor=north] at (12,\YL) {$n$};
    \node[anchor=north] at (13,\YL) {$\grecPos_{n,0}$};
    \draw[Bnr bracket] (12.5,\YLLL) -- node[above=4pt] {$\gB_{n,1}=2$} (6.5,\YLLL);
    \draw[Bnr bracket] (5.5,\YLLL) -- node[above=4pt] {$\gB_{n,2}=1$} (3.5,\YLLL);
    \draw[Bnr bracket] (2.5,\YLLL) -- node[above=4pt] {$\gB_{n,3}=1$} (0.5,\YLLL);
    \node[k record,label={0:$k$-record}] at (2,\YLL) {};
    \node[one record,label={0:one-record}] at (6,\YLL) {};
    \node[no record,label={0:node}] at (10,\YLL) {};
\end{tikzpicture}
\end{center}
\caption{An example of \(\mySeq{\gBnr}{p \ge 1}\) for \(n=12\).}
\label{fig:seg}
\end{figure}

Recall that \(\clock_{r,j}\eql\Exp(1)\), is the lapse of time between the
alarm clock of \(j\) rings for the \( (r-1)\)-st time and the \(r\)-th time.
Conditioning on \(\PRseq\),
for \(j \in (\grecPosnp, \grecPos_{n,p-1})\), 
we have
\begin{align}
    \E{\recordOnej}
    &
    =
    \p{
        \clock_{1,j}
        <
        \grecordnp
        \left|
            G_{k,j}
            > 
            \grecordnJ{p-1}
        \right.
    }
    .
    \label{eq:gBnr:dist:1}
\end{align}
Then the distribution of \(\gBnr{}\) 
conditioning on \(\PRseq\)
is simply that of
\begin{align}
    \Bin\left( 
        \grecPosnJ{p-1}-\grecPosnJ{p}-1, 
    \p{
        \clock_{1,j}
        <
        \grecordnp
        \left|
            G_{k,j}
            > 
            \grecordnJ{p-1}
        \right.
    }
    \right)
    ,
    \label{eq:gBnr:dist}
\end{align}
where \(\Bin(m,p)\) denotes a binomial \( (m,p)\) random variable.
When \(\grecordnJ{p-1}\) is small and \(\grecPosnJ{p-1}-\grecPosnp\) is large, this is roughly
\begin{equation}
    \Bin\left( \grecPosnJ{p-1}-\grecPosnJ{p}, \p{\clock_{1,j} < \grecordnJ{p} } \right)
    \eql 
    \Bin\left( \grecPosnJ{p-1}-\grecPosnJ{p}, 1 - \erm^{-\grecordnJ{p}} \right).
    \label{eq:pBnr:bin}
\end{equation}

Therefore, we first study a slightly simplified model. Let
\(\mySeq{\pclock_{r, j}}{r \ge 1, j \ge 1}\) be \iid{} \(\Exp(1)\) which are also independent from
\(\mySeq{\clockRJ{r}{j}}{r \ge 1, j \ge 1}\). Let
\begin{equation}
    \precordnj \eqd \iverson{\pclockj < \min\{\gclock_{k,i} : 1 \le i \le j \}},
    \qquad
    \precNumN \eqd
    \sum_{1 \le j \le n} \precordnj.
    \label{eq:precNumN}
\end{equation}
We say a node \(j\) is an \emph{alt-one-record} if \(\precordnj = 1\).
As in \eqref{eq:gBnr}, we can write
\begin{equation}
    \precNumN 
    = \sum_{1 \le j \le n} \precordnj
    = 
    \sum_{1 \le p} 
    \sum_{1 \le j}
    \,
    \iverson{\grecPosnJ{p-1} > j \ge \grecPosnJ{p}}
    \,
    \precordnj
    \eqd
    \sum_{1 \le p} \pBnr
    .
    \label{eq:pBr}
\end{equation}
Then conditioning on \(\PRseq\), \(\pBnr\) has exactly the distribution
as \eqref{eq:pBnr:bin}. Figure \ref{fig:seg:1} gives an example of \(\mySeq{\pBnr}{p\ge 1}\) for
\(n=12\). 
It shows the positions of alt-one-records, as well as the values and the summation ranges
of\(\mySeq{\pBnr}{p\ge 1}\). 

In the rest of this section, we will first prove the following proposition:
\begin{figure}[ht]
\begin{center}
\begin{tikzpicture}[scale=0.95]
    \tikzstyle{k record}=[circle,inner sep=5pt,draw=blue!50,fill=blue!20]
    \tikzstyle{one record}=[rectangle,inner sep=4pt,draw=black,fill=red!0]
    \tikzstyle{alt record}=[regular polygon,regular polygon sides=3,inner sep=2pt,draw=black,fill=red!0]
    \tikzstyle{no record}=[circle,inner sep=2pt,draw=black,fill=black]
    \tikzstyle{Bnr bracket}=[decoration={brace,mirror,raise=0pt},decorate]
    \pgfmathsetmacro{\YL}{-0.8}
    \pgfmathsetmacro{\YLL}{-1.7}
    \pgfmathsetmacro{\YLLL}{1.6}
    \pgfmathsetmacro{\ZL}{0.8}
    \draw[dashed] (-0.6,0) -- (0,0);
    \draw (0,0) -- (12,0);
    \draw[dashed] (12,0) -- (13,0);
    \draw[dashed] (-0.6,\ZL) -- (0,\ZL);
    \draw (0,\ZL) -- (12,\ZL);
    \draw[dashed] (12,\ZL) -- (13,\ZL);
    \foreach \x in {0,3,6}
    {
    \node[k record] at (\x,0) {};
    }
    \foreach \x in {0,2,3,4,6,9,11}
    {
    \node[one record] at (\x,0) {};
    }
    \foreach \x in {0, 1,5, 4,7,9,10,12}
    {
    \node[alt record] at (\x,\ZL) {};
    }
    \foreach \x in {0,...,12}
    {
        \node[no record] at (\x,0) {};
    }
    \foreach \x in {0,...,13}
    {
        \node[anchor=north] at (\x,-0.3) {$\x$};
    }
    \foreach \x in {-0.5, 2.5, 5.5, 12.5}
    {
        \draw[dotted] (\x,\YLLL) -- (\x,0);
    }
    \node[anchor=north] at (0,\YL) {$\grecPos_{n,3}$};
    \node[anchor=north] at (3,\YL) {$\grecPos_{n,2}$};
    \node[anchor=north] at (6,\YL) {$\grecPos_{n,1}$};
    \node[anchor=north] at (12,\YL) {$n$};
    \node[anchor=north] at (13,\YL) {$\grecPos_{n,0}$};
    \draw[Bnr bracket] (12.5,\YLLL) -- node[above=4pt] {$\pB_{n,1}=4$} (5.5,\YLLL);
    \draw[Bnr bracket] (5.5,\YLLL) -- node[above=4pt] {$\pB_{n,2}=2$} (2.5,\YLLL);
    \draw[Bnr bracket] (2.5,\YLLL) -- node[above=4pt] {$\pB_{n,3}=2$} (-0.5,\YLLL);
    \node[k record,label={0:$k$-record}] at (1,\YLL) {};
    \node[one record,label={0:one-record}] at (4,\YLL) {};
    \node[alt record,label={0:alt-one-record}] at (7,\YLL) {};
    \node[no record,label={0:node}] at (10.5,\YLL) {};
\end{tikzpicture}
\end{center}
\caption{An example of \(\mySeq{\pBnr}{p \ge 1}\) for \(n=12\).}
\label{fig:seg:1}
\end{figure}

The main part of the proof for \autoref{thm:rec:limit} consist of showing the following
\begin{proposition}
    \label{pro:B:main}
    For all fixed \(p \in \N\) and \(k \ge 2\),
    \begin{equation}
        \cL
        \left( 
        \left( 
        \frac{\pB_{n,1}}{\nscale}
        ,
        \dots
        ,
        \frac{\pB_{n,p}}{\nscale}
        \right)
        \right)
        \inlaw 
        \cL\left( 
        \left( 
            B_{1}
            ,
            \dots
            B_{p}
        \right)
        \right)
         ,
        \label{eq:B:coupling:1}
    \end{equation}
    which implies by the Cramér–Wold device that
    \begin{align}
        \cL
        \left( 
        \sum_{1 \le j \le p}
        \frac{\pB_{n,j}}{\nscale}
        \right)
        \inlaw
        \cL
        \left( 
        \sum_{1 \le j \le p}
        B_{j}
        \right)
        ,
        \label{eq:B:coupling}
    \end{align}
\end{proposition}
\noindent Then we can prove that \(p\) can be chosen large enough so that
\( { \sum_{p < j} \pBnj }/{ \nscale } \) is negligible. Thus,
\begin{align}
    \cL
    \left( 
    \frac{
        \precNumN
    }{
        \nscale 
    }
    \right)
    \eqd
    \cL
    \left( 
    \frac{
        \sum_{1 \le j}
        \pBnj
    }{
        \nscale 
    }
    \right)
    \inlaw 
        \cL
        \left( 
        \sum_{1 \le j}
        B_{j}
        \right)
        \eqd
    \cL\left( 
    \cB_{k}
    \right)
    \label{eq:precNumN:as}
    .
\end{align}
Following this, we can use a coupling argument to show that \(\recNumNOne/\nscale\) and
\(\precNumN/\nscale\) converge to the same limit, which finishes the proof of \autoref{thm:rec:limit}.
The section ends with some discussions on \(\cB_{k}\).

\subsection{Proof of \autoref{pro:B:main}}
\label{sec:limit:main}

To prove \eqref{eq:B:coupling}, we construct a coupling by defining all the random
variables being studied in one probability space. Let
\begin{equation} 
    \grecPosnp = \max\left\{ \ceil{U_{p} \left( \grecPosnJ{p-1}-1 \right)}, 1 \right\},
    \label{eq:Pnr:U} 
\end{equation} 
for \(p \ge 1\), where \(\mySeq{U_{j}}{j \ge 1}\) are \iid{} \(\Unif[0,1]\) random variables,
independent of everything else.
This is a valid coupling, since conditioning on \(\grecPosnJ{p-1}\), \(\grecPosnp\) is uniformly
distributed on \(\{1,\dots,\grecPosnJ{p-1}-1\}\).
Note that this implies that for all \(p \in \N\)
\begin{align}
    \frac{\grecPosnp}{n}
    \inas
    \prod_{1 \le s \le p} U_{s}
    .
    \label{eq:Pnr:U:1}
\end{align}
Then conditioning on \(\grecPosnSeq\), we generate the random variables
\(\mySeq{\clock_{r, j}}{r \ge 1, j \ge 1}\) according to their proper conditional distribution,
which determine \( (\gclockRJ{r}{j})_{r\ge 1, j \ge 1}\) and \(\grecordnSeq\).
Let \(\mySeq{\pclock_{r, j}}{r \ge 1, j \ge 1}\) be as before.

Recall that \(R_{m,1}\) is the minimum of \(m\) independent \(\Gam(k)\) random variables. Let
\(M(m,t) \eqd (R_{m,1}|R_{m,1} > t)\) for \(t \ge 0\).  Then conditioning on \(\grecPosnJ{p-1}\) and
\(\grecordnJ{p-1}\), \(\grecordnp\eql M(\grecPosnJ{p-1}-1, \grecordnJ{p-1})\).  The following lemma
allows us to describe the limit distribution of \(\grecordnp\) conditioning on \(\grecPosnJ{p-1}\) and
\(\grecordnJ{p-1}\).
\begin{lemma}
    \label{lem:M}
    Let \(k \in \N\).
    Assume that \(\frac{r_{m}}{m} \to 1\) and \(t \ge 0\).
    Then as \(m \to \infty\),
    \begin{equation}
        H_{m} \eqd r_{m}^{\frac{1}{k}} \cdot M\left(m, {t}{r_{m}^{-\frac{1}{k}}}\right) 
        \inlaw 
        \cL\left( 
            \left(t^{k} + k! E \right)^{\frac{1}{k}}  
        \right)
        ,
        \label{eq:M}
    \end{equation}
    where \(E \eql \Exp(1)\). In particular,
    \(
        m^{\frac{1}{k}}M(m,0) \inlaw 
        \cL
        \left( 
            (k! E)^{\frac{1}{k}}
        \right)
        .
    \)
    The convergence is also point-wise for the density functions.
    The lemma also holds if \(H_{m}\) by is replaced by
    \[
        H_{m}' \eqd r_{m}^{\frac{1}{k}} \cdot 
        \left( 
            1-
            \exp
            \left( 
                -
                M\left(m, {t}{r_{m}^{-\frac{1}{k}}}\right) 
            \right)
        \right)
        .
    \]
\end{lemma}

\begin{proof}
    We only prove the lemma for \(H_{m}\). Similar argument works for \(H_{m}'\).
    Let \(y_m=x/r_m^{\frac{1}{k}}\) and let \(s_{m}=t/r_m^{\frac{1}{k}}\).
    By \autoref{lem:gam:approx}, for all fixed \(x \ge t\), 
    \begin{equation}
        \begin{aligned}
        \p{H_{m} > x}
        &
        =
        \frac{\p{R_{m,1} \ge y_m}}{\p{R_{m,1} \ge s_{m}}}
        =
        \left( 
            \frac{
                \Gamma\left( k, y_{m} \right)
            }{
                \Gamma\left( k, s_{m} \right)
            }
        \right)^{m}
        \sim
        \exp \left(m \left( -\frac{y_m^{k}-s_{m}^{k}}{k!}\right)  \right)
        \\
        &
        \to
        \exp \left( - \frac{x^{k}-t^{k}}{k!} \right)
        =
        \p{
            \left( t^{k} + k! E  \right)^{\frac{1}{k}} > x
        }
        .
        \end{aligned}
        \label{eq:H:x}
    \end{equation}
    Using \eqref{eq:H:x} and the derivative formula for the incomplete gamma functions \cite[8.8.13]{NIST}, 
    it is straightforward to verify the point-wise convergence of density functions.
\end{proof}

The next step is to recursively apply \autoref{lem:M} to get a joint convergence in distribution
for \( (S_{n,1},\dots,S_{n,p})\) as well as \( (\Sp_{n,1},\dots,\Sp_{n,p})\), which are basically resacled versions of \(
(\grecordnJ{1},\dots,\grecordnJ{p})\) defined by
\begin{align}
    &
    \Lpnr \eqd \left(n \prod_{1 \le j < p} U_{j} \right)^{\frac{1}{k}}, 
    &
    &
    \Lnr \eqd \left(\grecPosnJ{p-1}-1 \right)^{\frac{1}{k}},
    \\
    &
    \Snr \eqd \Lpnr \grecordnp
    ,
    &
    &
    \Spnr \eqd \Lpnr (1-\erm^{-\grecordnp})
    . 
    \label{eq:Snr}
\end{align}
\begin{lemma}
    \label{lem:Snj:law}
    For all fixed \(p \in \N\) and \(k \in \{2,3,\dots\}\),
    \begin{align}
        \left( S_{n,1}, S_{n,2},\dots,S_{n,p} \right)
        \inlaw
        \cL
        \left( 
        \left( S_{1}, S_{2},\dots,S_{p} \right)
        \right)
        .
        \label{eq:Snj:law}
    \end{align}
    The convergence is also point-wise for the joint distribution functions.
    The lemma holds if \(\Snj\) is replaced by \(\Spnj\).
\end{lemma}

\begin{proof}
    We only prove the lemma for \(\Snj\). The same argument works for \(\Spnj\).

Let \(\cF=\sigma(\mySeq{U_{j}}{j \ge 1})\) denote the sigma algebra generated by \(\mySeq{U_{j}}{j \ge 1}\).
Throughout the proof of this lemma, we will condition on \(\cF\) and
treat \(\mySeq{U_{p}, \grecPosnp, \Lpnr, \Lnr}{p\ge0,n\ge1}\) as if they are deterministic numbers. 

Let \(f_{n, 1}(\cdot)\) and \(f_{1}(\cdot)\) denote the density functions of \(\SnOne\) and
\(S_{1}\) respectively.  For \(p>1\), let \(f_{n, p}(\cdot|y_{p-1})\) and \(f_{p}(\cdot|y_{p-1})\)
denote the density function of \(\Snr|S_{n,p-1}=y_{p-1}\), and \(S_{p}|S_{p-1}=y_{p-1}\)
respectively.  It follows from \autoref{lem:M} that for all \(y_{1} \ge 0\), \( f_{n,1}(y_{1}) \to
    f_{p}(y_{1})\), and for all \(y_{p} \ge 0\), \( f_{n,p}(y_{p}|y_{p-1}) \to f_{p}\left(
        y_{p}|y_{p-1} \right) .  \) Therefore, for all \(y_{1},\dots,y_{p} \in [0,\infty)^{p}\), as
\(n \to \infty\),
\begin{align}
    g_{n,p}(y_{1},\dots,y_{p})
    &
    \eqd
    f_{n,p}(y_{p}|y_{p-1})
    f_{n, p-1}(y_{p-1}|y_{p-2})
    \dots
    f_{n,1}(y_{1})
    \\
    &
    \to
    f_{p}(y_{p}|y_{p-1})
    f_{p-1}(y_{p-1}|y_{p-2})
    \dots
    f_{1}(y_{1})
    \eqd
    g_{p}(y_{1},\dots,y_{p})
    .
    \label{eq:Snr:to:3}
\end{align}
In other words, the joint density function of \( (\SnOne,\dots,\Snr)\) converges point-wise to the joint density function of 
\( (S_{1},\dots,S_{p})\) conditioning on \(\cF\). Thus, the lemma still holds without
conditioning on \(\cF\).
\end{proof}

One last ingredient needed is the next lemma which follows easily from Chernoff's bound, see, e.g.,
\cite[pp.~43]{reed02}.
\begin{lemma}
    \label{lem:wll}
    Let \(W_{m} \eql \Bin(m, p_{m}) \). If \(\ell_{m} p_{m} \to c \in (0, \infty)\) and \(m /\ell_{m} \to \infty\), then
    \(
        {\ell_{m} W_{m}}/{m}
        \inprob c.
    \)
\end{lemma}

\begin{proof}[Proof of {\autoref{pro:B:main}}]
As in the proof of \autoref{lem:Snj:law}, we condition on
\(\cF=\sigma(\mySeq{U_{j}}{j \ge 1})\) and treat \(\mySeq{U_{j}, \grecPosnj, \Lpnj}{j\ge0,n\ge1}\)
as deterministic numbers.
By \eqref{eq:pBnr:bin}, conditioning on \((\SpnOne,\dots,\Spnr)\), \( \pB_{n,1},\dots, \pBnr\)
are independent and for \(j \in \{1,\dots,p\}\),
\begin{align}
    \pBnj
    |
    (\SpnOne,\dots,\Spnr)
    \eql
    \Bin\left( \grecPosnJ{j-1}-\grecPosnJ{j}, \frac{\Spnj}{\Lpnj} \right)
    .
    \label{eq:pBnj:bin:one}
\end{align}
It follows from \eqref{eq:Pnr:U:1} and \autoref{lem:wll} that
\begin{equation}
    \left.
    \frac{\pBnj}{\nscale}
    \right|
    (\SpnOne,\dots,\Spnr)
    \inprob
    (1-U_{j}) \left(\prod_{1 \le s < j} U_s \right)^{1-\frac{1}{k}}
    \Spnj
    .
    \label{eq:pBnj:p}
\end{equation}
Now by \autoref{lem:Snj:law}, the joint density function of \((\SpnOne,\dots,\Spnr)\) converges
point-wise to that of \((S_{1},\dots,S_{p})\). Therefor, jointly, conditioning on \(\cF = \sigma(\mySeq{U_{j}}{j \ge 1})\),
\begin{equation}
        \left( 
        \frac{\pB_{n,1}}{\nscale}
        ,
        \dots
        ,
        \frac{\pB_{n,p}}{\nscale}
        \right)
        \inlaw 
        \cL
        \left( 
        \left( 
            B_{1}
            ,
            \dots
            B_{p}
        \right)
        \right)
        ,
    \label{eq:pBnj:as}
\end{equation}
where (see \eqref{eq:Br} and \eqref{eq:Sr})
\(
    B_{j} \eqd
    (1-U_{j}) \left(\prod_{1 \le s < j} U_s \right)^{1-\frac{1}{k}}
    S_{j}.
\)
Thus, the convergence also holds without conditioning on \(\cF\).
\end{proof}

\subsection{The leftovers}

In this section, we show that for \(p\) large enough, \(\sum_{s > p} B_{s}\), \(\sum_{s > p}
\pBns/\nscale\), and \(\sum_{s > p} \gBns/\nscale\) are all negligible. 
\begin{lemma}
    \label{lem:tail:Bj}
    For all \(k \in \{2,3,\dots\}\),  \(\varepsilon > 0\) and \(\delta > 0\), there exists \(p \in
        \N\) and \(n_0 \in \N\) such that for all \(n > n_{0}\),
    \begin{align}
        \p{\sum_{p < s} B_{s} \ge \varepsilon} < \delta
        ,
        \qquad
        \p{\frac{\sum_{j>p} \gBnj}{\nscale} \ge \varepsilon} 
        <
        \delta
        ,
        \qquad
        \p{\frac{\sum_{j>p} \gBnj}{\nscale} \ge \varepsilon} 
        <
        \delta
        .
        \label{eq:tail:B}
    \end{align}
\end{lemma}

\begin{proof}
    We only give the proof for \(\sum_{s > p} B_{s}\), since the other two can be dealt
    essentially the same way.

    Let \(U_{1}', E_{1}',U_{2}', E_{2}',\dots\) be independent random variables with
    \(U_{j}' \eql \Unif[0,1]\) and \(E_{j}' \eql \Exp(1)\).
    By the definition of \(B_{s}\) (see \eqref{eq:Br} and \eqref{eq:Sr}), we have 
    \begin{equation}
        B_{s} \preceq 
        B'_{s}
        \eqd
        \left( 
        \left( \prod_{1 \le j \le s} U_{j}' \right)
        \left(k! \sum_{1 \le j \le s} E_{s}' \right)
        \right)^{1/k}
        ,
        \label{eq:Bj:dom}
    \end{equation}
    i.e., \(B_{s}\) is stochastically dominated by \(B'_{s}\).
    Thus, we can prove the lemma for \(B'_{s}\) instead.
    Let \(W_{s}\) and \(W_{s}'\) be independent \(\Gam(s)\) random variables. Then
    \begin{equation}
        -\log\left(\prod_{1 \le j \le s} U_{j}' \right)
        \eql 
        W_{s},
        \qquad
        \sum_{1 \le j \le s} E_{j}' \eql W_{s}'
        .
        \label{eq:gamma:j}
    \end{equation}
    It is well known that \(\E{\left( W_{s} -s \right)^{4}} = 3s^{2}+6s\) \cite[pp.~339]{johnson94}.
    It follows from Markov's inequality that for \(s \ge 1\),
    \begin{equation}
        \p{|W_{s}-s| \ge \frac{s}{2}}
        \le
        \frac{\E{\left( W_{s} -s \right)^{4}}}{s^{4}/16}
        =
        \frac{3s^{2}+6s}{s^{4}/16}
        =
        \frac{9s^2}{s^{4}/16}
        \le
        \frac{144}{s^2}
        .
        \label{eq:Wj}
    \end{equation}
    Therefore
    \begin{align}
        \p{\left(B'_{s}  \right)^{k} \ge k! \frac{3}{2}s \erm^{-s/2}}
        &
        \le
        \p{\prod_{1 \le j \le s} U_{j}'
            \ge \erm^{-s/2}}
        +
        \p{\sum_{1 \le j \le s} E_{j}' \ge \frac{3}{2}s}
        \\
        &
        =
        \p{W_{s} \le \frac{s}{2}}
        +
        \p{W_{s}' \ge \frac{3s}{2}}
        =
        \bigO{\frac{1}{s^2}}
        .
        \label{eq:Bj:upper}
    \end{align}
    We are done since 
    \[
        \sum_{s > p} \frac{1}{s^2} = O(p^{-1}), 
        \qquad
        \sum_{s > p} \left( k!  \frac{3}{2} s \erm^{-s/2} \right)^{\frac{1}{k}} =
        \bigO{-\erm^{\frac{p}{4k}}}.
        \qedhere
    \]
\end{proof}

\subsection{Finishing the proof Theorem of \ref{thm:rec:limit}}

By \autoref{lem:tail:Bj}, the contribution of  \(\sum_{s > p} B_{s}\) and \(\sum_{s > p}
\pBns/\nscale\) in \(\sum_{s > 1} B_{s}\) and \(\sum_{s > 1} \pBns/\nscale\) respectively can be
made arbitrarily small by choosing \(p\) large enough.  Thus, it follows from \autoref{pro:B:main}
that \({ \precNumN }/{ \nscale } \inlaw \cL\left( \cB_{k} \right)\) as we claimed.

Now we fill the gap between \(\precNumN\) and \(\recNumNOne\) by the following lemma,
from which \autoref{thm:rec:limit} follows immediately.
\begin{lemma}
    \label{lem:gap}
    Let \(k \in \{2,3,\dots\}\).
    There exists a coupling such that
    \begin{align}
        \label{eq:gap}
        \frac{
            \precNumN - \recNumNOne
        }{
            \nscale
        }
        \inprob
        0
        .
    \end{align}
\end{lemma}

\begin{proof}
Recall that \(\mySeq{\pclock_{i,j}}{i \ge 1, j \ge 1}\)
are \iid{} \(\Exp(1)\) random variables that we used, together with \(\mySeq{\grecPosnj,
\grecordnj}{j\ge0}\) to define \(\precNumN\). Now we modify \(\mySeq{\clock_{i,j}}{i \ge 1,
j \ge 1}\) by letting \(\clock_{i,j} = \pclock_{i,j}\) for all \(i \in \N\) and
\(j \not \in \{\grecPosnJ{j}\}_{j \ge 0}\), unless there is a
\emph{discrepancy}, i.e., if for some \(p \ge 1\),
\begin{align}
    \grecPosnJ{p-1} < j < \grecPosnJ{p}
    ,
    \qquad
    \text{and}
    \qquad
    \sum_{i=1}^{k} \pclock_{j, i} < \grecordnp
    .
    \label{eq:disc}
\end{align}
This is a valid coupling since it does not change the distribution of \(\mySeq{\gBnj}{j \ge 1}\).

Let \(J_{n, p}\) denote the number of discrepancies between \(\grecPosnJ{p-1}\) and \(\grecPosnJ{p}\), i.e.,
\begin{align}
    J_{n,p}
    &
    =
    \sum_{j \ge 1}
    \iverson{
        \grecPosnJ{p-1} < j < \grecPosnJ{p}
    }
    \iverson{
        \grecordnp
        >
        \sum_{1 \le i \le k}
        \pclock_{i,j}
    }
    .
    \label{eq:disc:1}
\end{align}
By the definition \eqref{eq:gBnr} and \eqref{eq:pBr},
with the above coupling, for all fixed \(p \in \N\),
\begin{align}
    |\recNumNOne-\precNumN|
    \le
    \sum_{1 \le j \le p}
    J_{n,j}
    +
    2 \recNumNk
    +
    \sum_{j > p} \gBnr
    +
    \sum_{j > p} \pBnr
    .
    \label{eq:disc:7}
\end{align}
By \autoref{thm:rec:mean}, we have \(\recNumNk/\nscale \inprob 0\).  It follows from 
\autoref{lem:tail:Bj} that by choosing \(p\) large enough, the last two
terms of the right-hand-side of \eqref{eq:disc:7} divided by \(\nscale\) are negligible. Thus, it suffices to 
show that \(\sum_{1 \le j \le p} J_{n,j}/n^{1-1/k} \inprob 0\).

Conditioning on \(\mySeq{\grecordnj, \grecPosnj}{n \ge 1, j \ge 0}\),
\begin{align}
    J_{n,p}
    \eql
    \Bin
    \left( 
        \grecPosnJ{p-1} - \grecPosnJ{p} -1,
        \p{G_{k} < \grecordnp}
    \right)
    ,
    \label{eq:disc:2}
\end{align}
where \(G_{k}\eql\Gam(k)\).
Therefore, it follows from 
the series expansion of the incomplete gamma function \cite[8.7.3]{NIST} 
that
\begin{align}
    \E{J_{n,p}~|~\mySeq{\grecordnj, \grecPosnj}{n \ge 1, j \ge 0}}
    &
    \le
    (\grecPosnJ{p-1} - \grecPosnJ{p})
    \cdot
    \left(1-\frac{\Gamma(k, \grecordnp)}{\Gamma(k)} \right)
    \\
    &
    \le
    \grecPosnJ{p-1} \grecordnp^{k}
    =
    \frac{\grecPosnJ{p-1}}{(\Lpnr)^{k}}(\Snr)^{k}
    \inlaw
    S_{p}^{k}
    ,
    \label{eq:disc:3}
\end{align}
where the converges follows from \eqref{eq:Pnr:U:1} and \autoref{lem:Snj:law}.
By the definition \eqref{eq:Sr},
\(
S_{p}^{k} \preceq k! G_{p}\).
Thus
for all fixed \(p \in \N\), \(\sup_{n \ge 1} \E{J_{n,p}} < \infty\) and
\(
    \sum_{1 \le i \le p} {J_{n,i}}/{\nscale} \inprob 0.
\)
\end{proof}

\subsection{The density of \(\cB_{k}\)}
\label{sec:B:k}

\begin{lemma}\label{lem:B:density}
    For all \(k \in \{2,3,\dots\}\) the random variable \(\cB_{k}\) defined in \eqref{eq:cBk} has a density function.
\end{lemma}

\begin{proof}
The random variable $\cB_k$ can be written as
\(
  \sum_{1 \le p} (a(p) + b(p) E_1 )^{1/k},
\)
where $a(p)$ and $b(p)$ are (complicated) non-negative functions of the random vector $Z \eqd
(U_1,U_2,E_2, U_3, E_3,\dots)$.  Conditioning on $Z$, $\cB_k$ has a density provided that $b(p) \ne 0$
for some $p$.  Thus, a sufficient condition for $\cB_k$ to have a density is that $\p{b(1)=0} = 0$,
which is obvious since \(b(1)=(1-U_{1})^{1/k}k!\)
\end{proof}

It is not easy to see what the density function of \(\cB_{k}\) should be like analytically. But
through simulation, it is obvious that \(\cB_{k}\) has a density very close to that of the normal
distribution \(\cN(\e \cB_{k}, \sqrt{\V{\cB_{k}}})\), see \autoref{fig:Bk}. It is perhaps not so
surprising. Once the positions and values of \(k\)-records are fixed, \(\recNumNOne\) is
simply a sum of independent indicator random variables, which often gives rise to the normal
distribution. Comparing \autoref{fig:k2} with the simulation result for \(\recNumN\) with \(k=2\)
shown in \autoref{fig:k2-sim}, we see that \(\cB_{k}\) is indeed the limit distribution of
\(\recNumN\).

\begin{figure}[ht]
\begin{subfigure}{.5\textwidth}
  \centering
  \includegraphics[width=\linewidth]{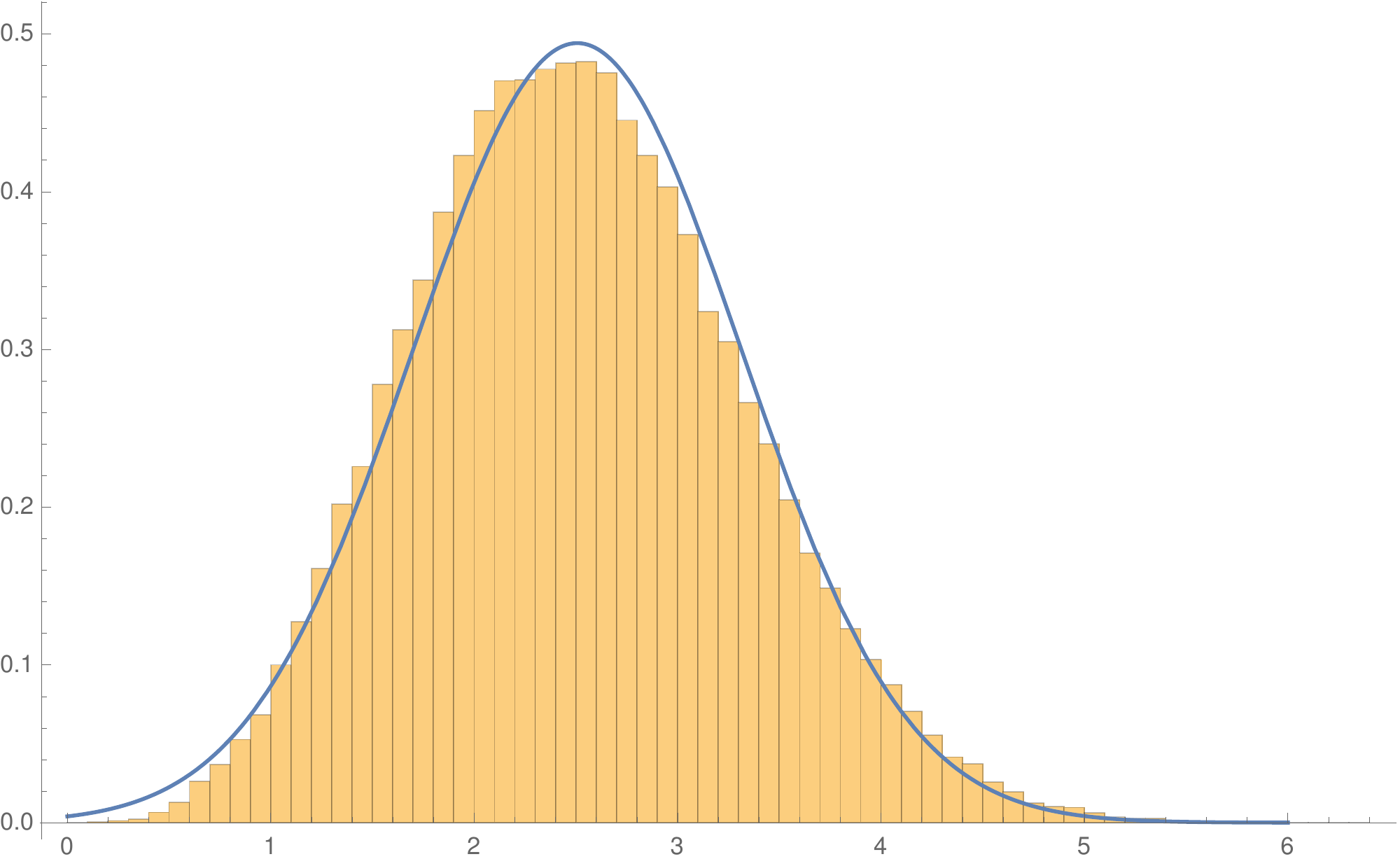}
  \caption{\(k=2\)}
  \label{fig:k2}
\end{subfigure}%
\begin{subfigure}{.5\textwidth}
  \centering
  \includegraphics[width=\linewidth]{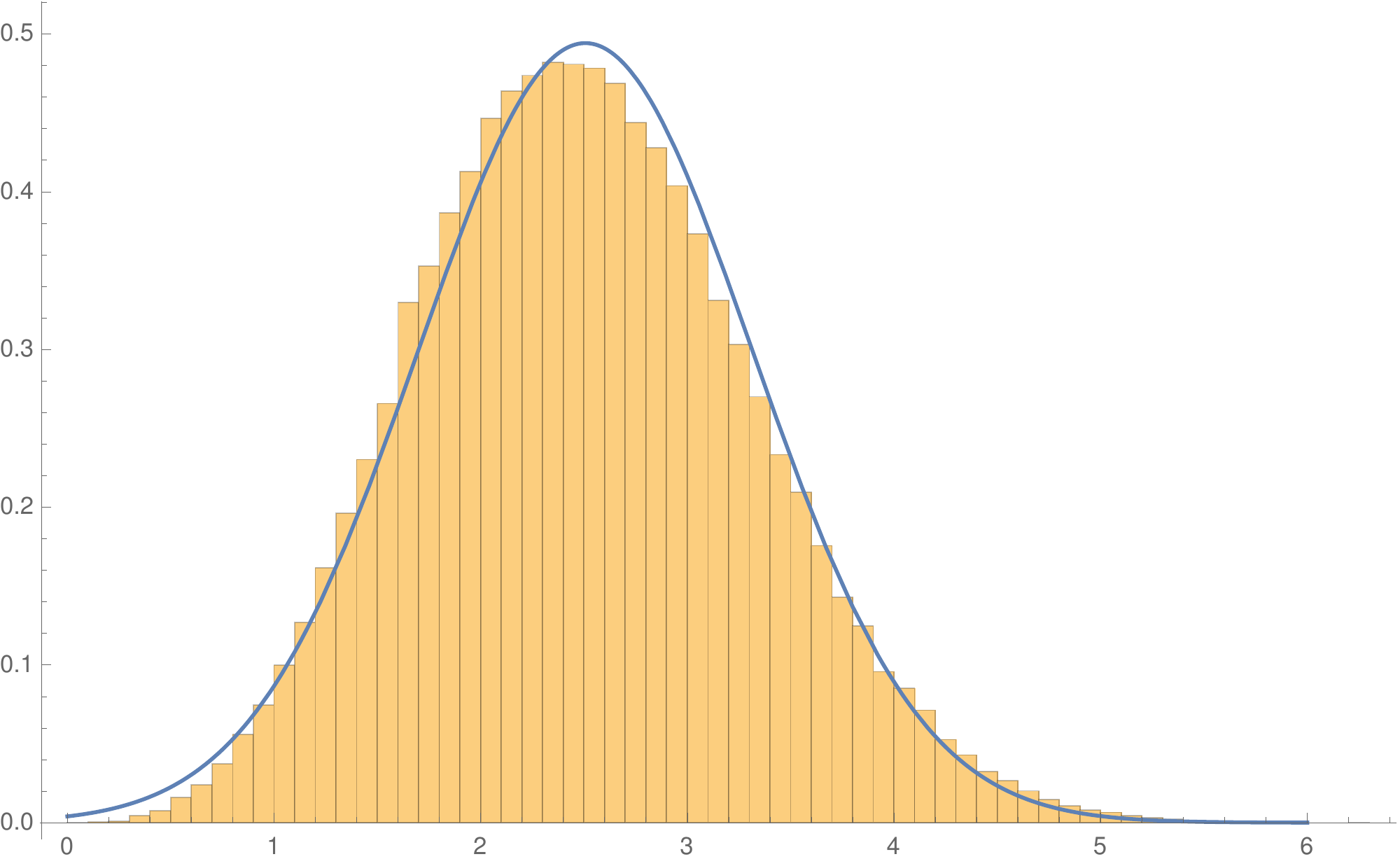}
  \caption{\(k=3\)}
  \label{fig:k3}
\end{subfigure}
\\
\begin{subfigure}{.5\textwidth}
  \centering
  \includegraphics[width=\linewidth]{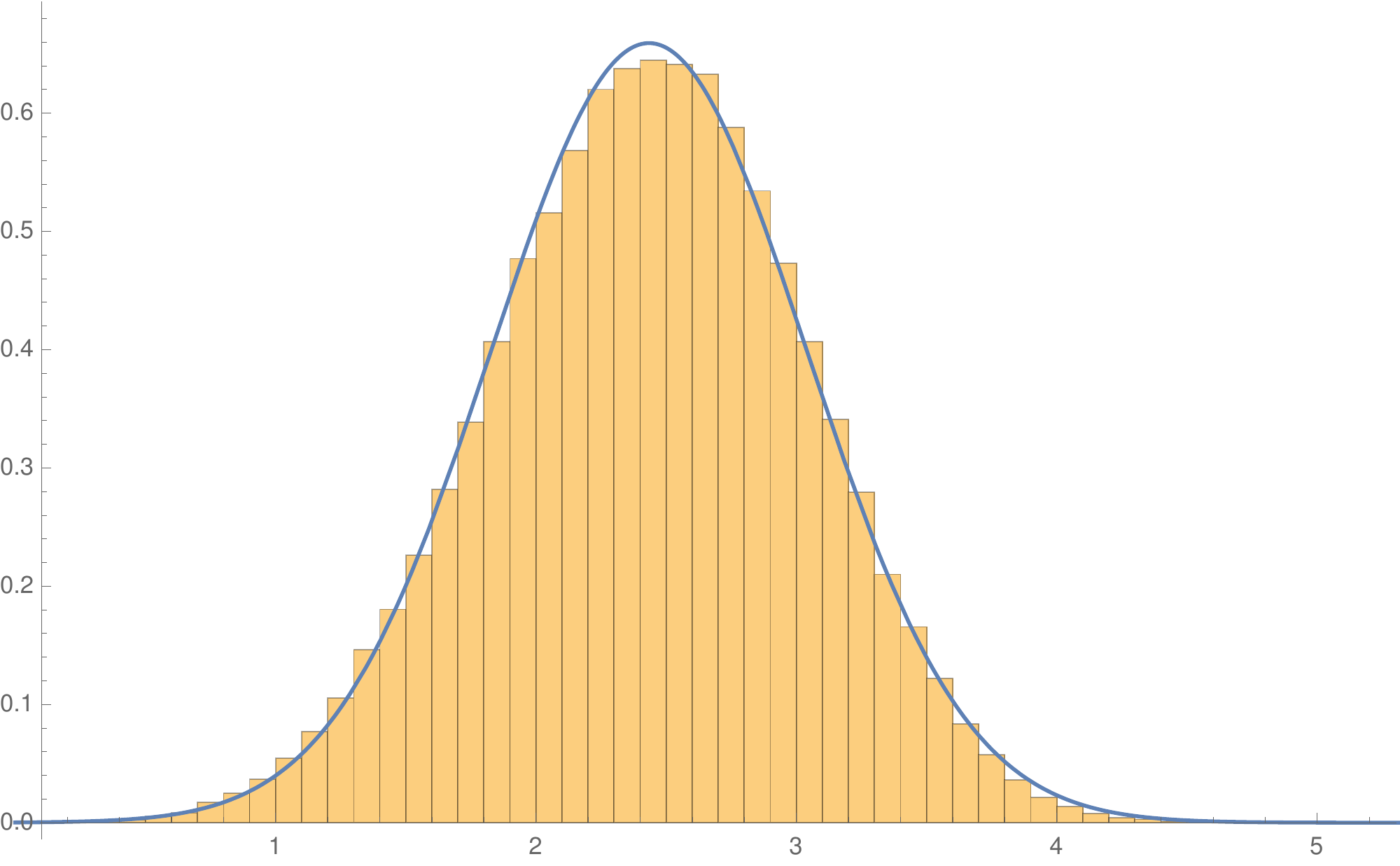}
  \caption{\(k=4\)}
  \label{fig:k4}
\end{subfigure}%
\begin{subfigure}{.5\textwidth}
  \centering
  \includegraphics[width=\linewidth]{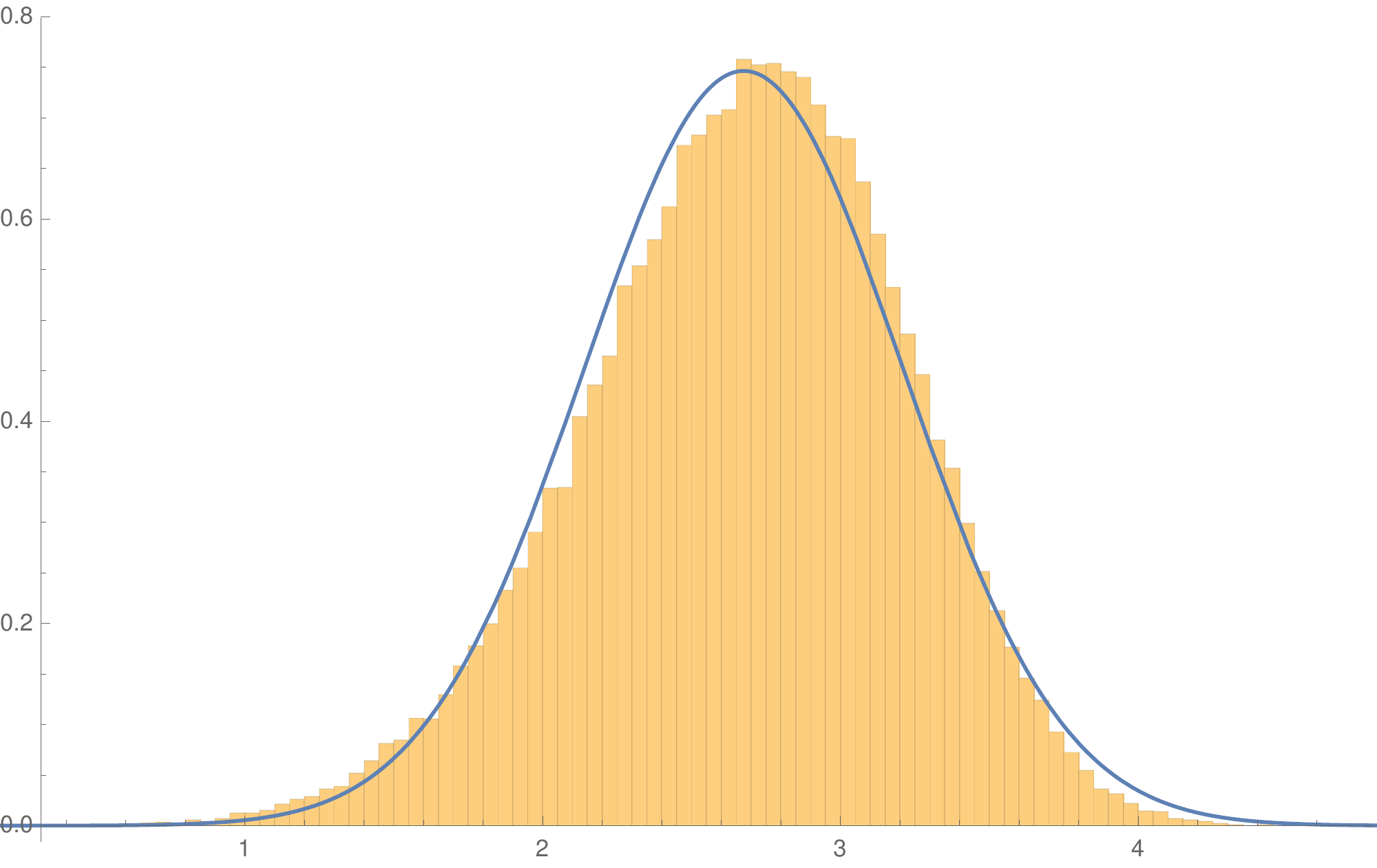}
  \caption{\(k=5\)}
  \label{fig:k5}
\end{subfigure}
\caption[Simulations for \(\cB_{k}\)]{Histograms of \(10^5\) samples of \(\cB_{k}\) for
    \(k=2,\dots,5\). The blue curves represent the density functions of \(\cN(\e \cB_{k},
        \sqrt{\V{\cB_{k}}})\).}
\label{fig:Bk}
\end{figure}

\begin{figure}[ht]
    \centering
  \centering
  \includegraphics[width=\linewidth]{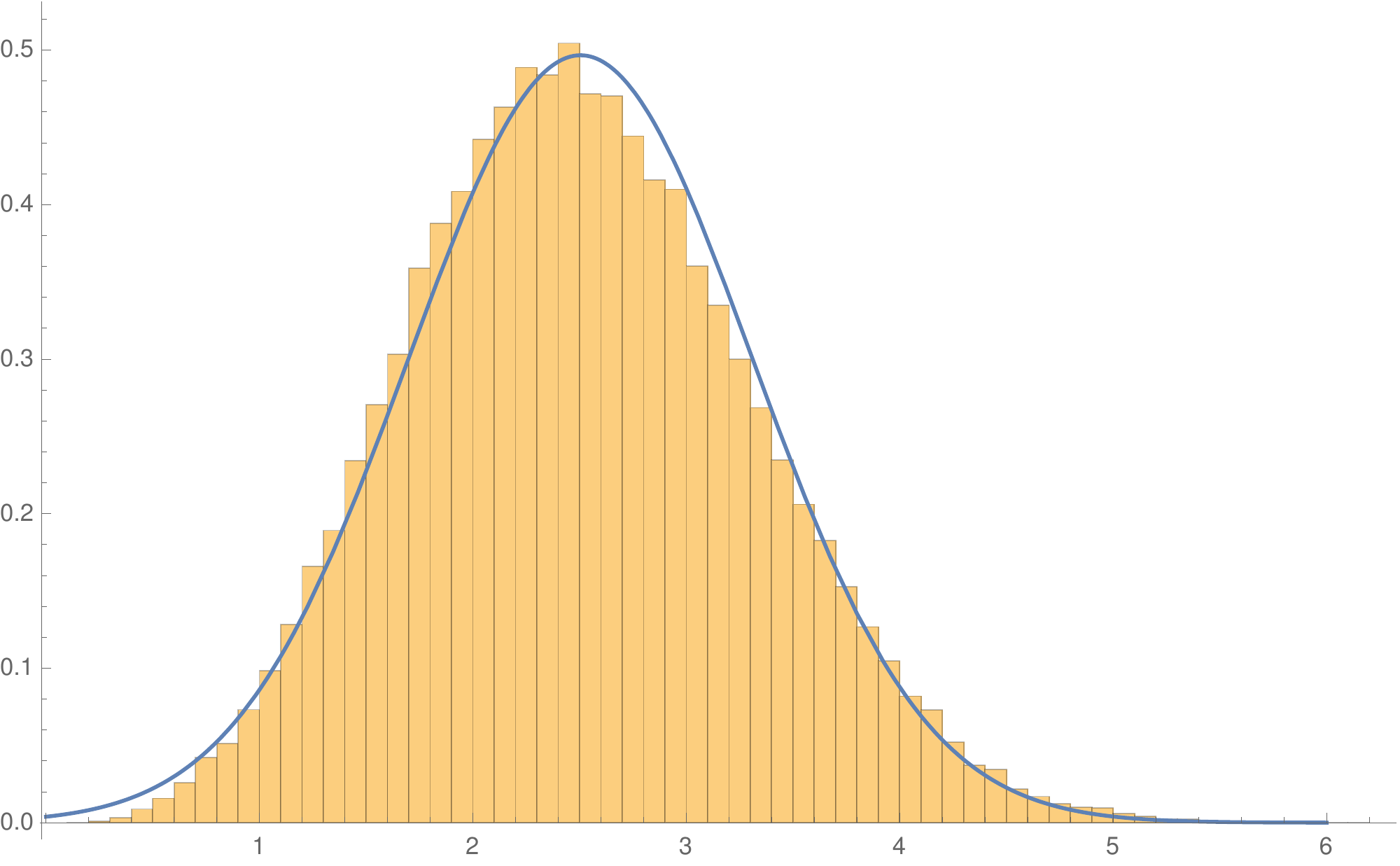}
  \caption[Simulation for \(\recNumN{2}\)]{Simulation for \(\recNumN\) with \(k=2\),
      \(n=2^{17}\) and \(60000\) samples, after rescaled by \(\sqrt{n}\). The blue curve represents the
      density function of a normal distribution with the empirical mean and variance.}
  \label{fig:k2-sim}
\end{figure}

\section{Some extensions}\label{sec:tree}

\subsection{A lower bound and an upper bound for general graphs}

Let \(\cG_{n}\) be the set of rooted graphs with \(n\) nodes.  It is obvious that
\(\Pathn\) is the easiest to cut among all graphs in \(\cG_{n}\).  We formalize this by the
following lemma:

\begin{lemma}
    \label{lem:universal}
    Let \(k \in \N\).
    For all \(\Graphn \in \cG_{n}\),
    \(
        \cutNumN \eqd \kcut(\Pathn) \preceq \kcut(\Graphn).
    \)
    Therefore,
    \begin{equation}
        \min_{\Graphn \in \cG_{n}} \e{\kcut(\Graphn)} \ge \e{\cutNumN} \sim
        \left\{
        \begin{array}{*2{>{\displaystyle}l}}
        \frac{(k!)^{\frac{1}{k}}}{k-1} \Gamma\left( \frac{1}{k} \right)
        n^{1-\frac{1}{k}}
        \qquad
        \qquad
        & 
        (k \ge 2),
        \\
        \log n
        &
        (k=1)
        ,
        \end{array}
        \right.
        \label{eq:P:exp}
    \end{equation}
    by Theorem \ref{thm:rec:mean}.
\end{lemma}

The most resilient graph is obviously \(\Kn\), the complete graph with \(n\) vertices. 
Thus, we have the following upper bound:
\begin{lemma}
    \label{lem:universal:up}
    Let \(k \in \N\).
    \begin{enumerate}[(i)]
        \item 
    Let \(Y \eql \Gam(k)\), \(Z \eql \Poi(Y)\), and \(W \eql Z \wedge k\), i.e., \(W \eql \min\{Z,
        k\}\).  Then
    \begin{equation}
        \cL
        \left( 
        \frac{
            \kcut(\Kn)
        }{
            n
        }
        \right)
        \inlaw
        \cL
        \left( 
        \E{W|Y}
        \right)
        =
        \cL
        \left( 
        \frac{\Gamma (k+1,Y)-e^{-Y} Y^{k+1}}{k!}+k
        \right)
        ,
        \label{eq:star:dist}
    \end{equation}
    where \(\Gamma(\ell,z)\) denotes the upper incomplete gamma function.
    Note that when \(k=1\), the right-hand-side is simply \(\Unif[0,1]\).
    \item
    For all \(\Graphn \in \cG_{n}\),
    \(
        \kcut(\Graphn) \preceq \kcut(\Kn).
    \)
    Therefore,
    \begin{equation}
        \max_{\Graphn \in \cG_{n}}
        \e{\kcut(\Graphn)} 
        \le
        \e{\kcut(\Kn)}
        \sim
        k
        \left(
        1-\frac{1}{2^{2 k}} \binom{2 k}{k}
        \right)
        n
        .
        \label{eq:star}
    \end{equation}
    \end{enumerate}
\end{lemma}

\begin{proof}
    Let \(S_{n}\) be the tree of \(n\) nodes with one root and \(n-1\) leaves. Obviously
    \(\kcut(\Kn) \eql \kcut(S_{n})\).
    Let \(Y\) be the time when the root is removed.
    Let \(W_{1,n},\dots,W_{n-1,n}\) be the number of cuts each leaf receives by this time.
    Conditioning on the event \(Y=y\), \(W_{1,n},\dots,W_{n-1,n}\) are \iid with \(W_{i,n} \eql
        Z_{i}\wedge k\), where \(Z_{i}\eql \Poi(y)\). In other words,
    conditioning on \(Y=y\), by the law of large numbers,
    \begin{equation}
        \frac{
            \kcut(S_{n})
        }{
            n
        }
        =
        \frac{
            k + \sum_{i=1}^{n-1} W_{i,n}
        }{
            n
        }
        \inas
        \E{
            Z_{1}
        }
        ,
        \label{eq:star:as}
    \end{equation}
    from which \eqref{eq:star:dist} and \eqref{eq:star} follow immediately.
\end{proof}

\subsection{Path-like graphs}

If a graph \(\Graphn\) consists of only long paths, then the limit distribution \(\kcut(\Graphn)\)
should be related to \(\cB_{k}\), the limit distribution of \(\kcut(\Pathn)/n^{1-\frac{1}{k}}\)
(see \autoref{thm:rec:limit}).  We give two simple examples with \(k \in \{2,3,\dots\}\).

\begin{example}[Long path]
    Let \(\mySeq{\Graphn}{n \ge 1}\) be a sequence of rooted graphs such that \(\Graphn\) contains a
    path of length \(m(n)\) starting from the root with
    \(n-m(n) = o(n^{1-\frac{1}{k}})\). Since it takes at most \(k (n-m(n))\) cuts to remove all the
    nodes outside the long path, 
    \[
        \kcut(P_{m(n)}) \preceq \kcut(\Graphn) \preceq \kcut(P_{m(n)}) + k \smallo{n^{1-1/k}}.
    \]
    Thus, by \autoref{lem:universal}, this implies that
    \(\kcut(\Graphn)/n^{1-\frac{1}{k}}\) converges in distribution to \(\cB_{k}\).
\end{example}

\newcommand{\treeCu}{\mathbb T^{(\ell)}_{n}}
\begin{example}[Curtain]
    Let \(\ell \ge 2\) be a fixed integer. Let \(\treeCu\) be a graph that consists of only \(\ell\) paths
    connected to the root, with the first \(\ell-1\) of them having length \(\ceil*{\frac{n-1}{\ell}}\).
    We call \(\treeCu\) an \(\ell\)-curtain. It is easy to see that cutting \(\treeCu\) is very similar to
    cutting \(\ell\) separated paths of length \(\ceil*{\frac{n}{\ell}}\). Therefore, we can show
    that
    \begin{equation}
        \frac{\kcut(\treeCu)}{\left( n/\ell \right)^{\frac{1}{k}}}
        \inlaw
        \sum_{j=1}^{\ell}
        \cB_{k}^{[j]},
        \label{eq:curtain}
    \end{equation}
    where \(\cB_{k}^{[1]},\dots,\cB_{k}^{[\ell]}\) are \iid copies of \(\cB_{k}\).
\end{example}

\subsection{Deterministic and random trees}\label{sec:logtrees}

The approximation given in \autoref{lem:mean:i} can be used to compute the expectation of \(k\)-cut
numbers in many deterministic or random trees. We give four examples: complete binary trees, split trees,
random recursive trees, and Galton-Watson trees.

\newcommand{\treeBi}{\mathbb{T}^{\mathrm{bi}}_{n}}

\subsubsection{Complete binary trees}\label{exam:binary}

Let \(\treeBi\) be a complete binary tree of with \(n = 2^{m+1}-1\) nodes, i.e., its height is
\(m\).  Recall that $I_{r,i+1}$ in \autoref{lem:mean:i} is the indicator that a node in \(\Pathn\)
at depth $i$ is an $r$-record.  Since the probability of a node being an \(r\)-record only depends
on its depth, it follows from \autoref{lem:mean:i} that
\begin{equation}
    \label{eq:bin}
    \begin{aligned}
    \e{\kcut_{r}(\treeBi)}
    =
    \sum_{i=0}^{m} 2^{i} \e{I_{r,i+1}}
    \sim
    \frac{(k!)^{\frac{r}{k}} }{k}
    \frac{\Gamma\left( \frac{r}{k} \right)}{\Gamma(r)}
    \frac{2^{m + 1}}{m^{\frac{r}{k}}}
    .
    \end{aligned}
\end{equation}
Thus, only the one-records matter as in the case of \(\Pathn\) and
\begin{equation}
    \e{\kcut(\treeBi)}
    \sim
    \e{\kcut_{1}(\treeBi)}
    \sim
    \frac{(k!)^{\frac{1}{k}}{\Gamma\left( \frac{1}{k} \right)} }{k}
    \frac{2^{m+1}}{m^{\frac{1}{k}}}
    \sim
    \frac{(k!)^{\frac{1}{k}}{\Gamma\left( \frac{1}{k} \right)} }{k}
    \frac{n}{\left(\log_{2}n\right)^{\frac{1}{k}}}
    .
    \label{eq:bin:1}
\end{equation}
The limit distribution of \(\kcut(\treeBi)\) has been found in our follow-up paper
\cite{2018arXiv181105673S}.

\subsubsection{Split trees}\label{sec:split}

\newcommand{\treeSp}{\mathbb T^{\mathrm{sp}}_{n}}
\newcommand{\badSp}{\mathbb B^{\mathrm{sp}}_{n}}
\newcommand{\goodSp}{\mathbb X^{\mathrm{sp}}_{n}}

Split trees were first defined by \citet{MR1634354} to encompass many families of trees that are
frequently used in algorithm analysis, e.g., binary search trees and tries.  Its exact construction
is somewhat lengthy and we refer readers to either the original algorithmic definition in
\cite{MR2878784} or the more probabilistic version in \cite[Section 2]{Cai014}.  

Very roughly speaking, $\treeSp$ is constructed by first distributing randomly \(n\) balls among the
nodes of an infinite \(b\)-ary tree and then removing all subtrees without balls.  Each node
in the infinite \(b\)-ary tree is given a random non-negative split vector
\(\cV=(V_{1},\dots,V_{b})\), satisfying \(\sum_{i=1}^b V_i = 1\), drawn independently from the same
distribution. These vectors affect how balls are distributed.  

In the study of split trees, the following condition of \(\cV\) is often assumed:
\begin{myCond} \label{cond:split}
    The split vector \(\cV\) is permutation invariant. Moreover, \(\p{V_{1} = 1} = 0\), \(\p{V_{1}=0} = 0\), and that \(-\log(V_{1})\) is non-lattice.
\end{myCond}
\noindent
\citet[Theorem~1.1]{MR2878784} showed that 
, assuming condition \ref{cond:split},
there exists a constant \(\alpha\) such that \(
\e N \sim \alpha n\), where \(N\) is the random number of nodes in \(\treeSp\).

In the setup of split trees (and other random trees), we obtain \(\kcut(\treeSp)\) by choosing a
random split tree first and then carry out the \(k\)-cut process conditioning on the tree.
\citet[Theorem~1.1]{holmgren11} showed that condition \ref{cond:split} implies that
\(\kcut_{k}(\treeSp)\) converges to a weakly \(1\)-stable distribution after normalization, and that
\( \e \kcut_{k}(\treeSp) \sim {\mu \alpha n}/{\log n} \),
where \(\mu \eqd b \E{-V_{1} \log V_{1}}\).
We extend this result as follows:
\begin{lemma}
    \label{thm:split:mean}
    Assuming condition \ref{cond:split}, we have
    \begin{align}
        &
        \E{\kcut_r(\treeSp)}  
        \sim
        \frac{(k!\mu )^{\frac{r}{k}}}{k} \frac{\Gamma\left( \frac{r}{k}
            \right)}{\Gamma(r)}\frac{\alpha n}{ (\log n)^\frac{r}{k}}
        ,
        & 
        (1 \le r \le k),
        \label{eq:split:mean:k}
        \\
        &
        \E{\kcut(\treeSp)}  
        \sim
        (k!\mu )^{\frac{1}{k}} \frac{\Gamma\left( \frac{1}{k}
            \right)}{k}\frac{\alpha n}{ (\log n)^\frac{1}{k}}
        \label{eq:split:mean}
        .
    \end{align}
\end{lemma}

\begin{proof}
We say a node $v$ is \emph{good} if it has depth \(d(v)\) where
\(
    \left|d(v)-\frac{1}{\mu}\log n \right| \leq \log^{0.6} n,
\)
otherwise we say it is \emph{bad}.
Let \(\badSp\) be the number of bad nodes in \(\treeSp\).
By \cite{MR2878784}[Theorem~1.2],
\(
    \e{\badSp} = 
    \bigO{
        {n}/{(\log n)^3}
    }
    .
\)
Thus, the number of \(r\)-records in bad nodes is negligible and it suffices to prove the lemma for
good nodes.  By Lemma~\ref{lem:mean:i} and the definition of good nodes, we have
\begin{align*}
    \E{\kcut(\treeSp)|\treeSp} 
    =
    & 
    (N-\badSp)  \frac{(k!)^{\frac{r}{k}}}{k} \frac{\Gamma\left( \frac{r}{k} \right)}{\Gamma(r)}
    \left[ \frac{\log n}{\mu}+O(\log^{0.6}n)\right]^{-\frac{r}{k}}(1+O(\log^{-\frac{1}{2k}} n))\\ 
    =
    & 
    (N-\badSp)
    \frac{(k!\mu)^{\frac{r}{k}}}{k} \frac{\Gamma\left( \frac{r}{k}
        \right)}{\Gamma(r)}\frac{1}{\left(\log n\right)^{\frac{r}{k}}}(1+O(\log^{-\frac{1}{2k}} n)),
\end{align*}
from which the lemma follows by taking expectation and using that \(\e N\sim\alpha n\).
\end{proof}

\subsubsection{Random recursive trees}

\newcommand{\treeRR}{\mathbb{T}^{\mathrm{rr}}}
\newcommand{\treeRRn}{\treeRR_{n}}

A random recursive tree \(\treeRRn\) is random tree of \(n\) nodes constructed recursively as
follows: let \(\treeRR_{1}\) be the tree of a single node labeled \(1\); given \(\treeRR_{n-1}\),
choose a node in \(\treeRR_{n-1}\) uniformly at random and attach a node labeled \(n\) to the selected node as a child, which
gives \(\treeRRn\).  \textcite{meir74} introduced this model and showed that \(\e \kcut_{k}(\treeRRn)\sim
n /\log n\) and that \(\kcut_{k}(\treeRRn)\) concentrates around its mean. \textcite{drmota09} and
subsequently \textcite{iksanov07} proved \(\kcut(\treeRRn)\) converges weakly to a stable law after
proper shifting and normalization.

The intuition behind \(\e \kcut_{k}(\treeRRn)\sim n /\log n\) is simply that almost all nodes in
\(\treeRRn\) are at depth around \(\log n\).  We say a node \(v\) in \(\treeRRn\) is \emph{good} if
\(|d(v)-\log(n)|\le\log(n)^{0.9}\); otherwise we say it is \emph{bad}.  The following lemma shows
that there are very few bad nodes in expectation:
\newcommand{\badRR}{\mathbb B^{\mathrm{rr}}_{n}}
\begin{lemma}\label{lem:bad:rrt}
    Let \(\badRR\) be the number of bad nodes in \(\treeRRn\), then
    \(\e \badRR = \bigO{{n}/{\log(n)^{3}}}\).
\end{lemma}
\begin{proof}
Let \(h(\treeRRn)\) be the height of \(\treeRRn\). By \cite[6.3.2]{drmota2009}
\begin{equation}\label{eq:rrt:height}
    \p{|h(\treeRRn)-\ee \log(n)|> \eta}=\bigO{\ee^{-c \eta}}
    ,
\end{equation}
for some constant \(c\).
Thus, we can choose some constant \(K\) large enough and ignore the nodes of depth greater than
\(K \log(n)\).
Let \(w_{d}(\treeRRn)\) be the number of nodes at depth \(d\) in
\(\treeRRn\). By \cite[Equation 3]{MR2291961}
\begin{equation}\label{eq:rrt:profile}
    \E{w_{d}(\treeRRn)}
    =
    \frac{\log(n)^{d}}{\Gamma(1+d/\log(n))d!}\left( 1+ \bigO{\log(n)^{-1}} \right)
    ,
\end{equation}
uniformly for all \(n \ge 3\) and \(1 \le d \le K \log(n)\), for all \(K \ge 1\).
Thus, the lemma follows by summing both sides of \eqref{eq:rrt:profile} over integers \(d\) in \([1,
\log(n)-\log(n)^{0.9}] \cup [\log(n)+\log(n)^{0.9}, K \log(n)].\)
\end{proof}

Thus, by exactly the same argument of \autoref{thm:split:mean}, we get:
\begin{lemma}\label{lem:mean:rrt}
    We have
    \begin{align}
        &
        \E{\kcut_r(\treeRRn)}  
        \sim
        \frac{(k!)^{\frac{r}{k}}}{k} \frac{\Gamma\left( \frac{r}{k}
            \right)}{\Gamma(r)}\frac{n}{ (\log n)^\frac{r}{k}}
        ,
        & 
        (1 \le r \le k),
        \label{eq:rrt:mean:k}
        \\
        &
        \E{\kcut(\treeRRn)}  
        \sim
        (k!)^{\frac{1}{k}} \frac{\Gamma\left( \frac{1}{k}
            \right)}{k}\frac{n}{ (\log n)^\frac{1}{k}}
        \label{rrt:split:mean}
        .
    \end{align}
\end{lemma}

\begin{remark}
    We have not tried to find the limit distributions for \(\kcut(\treeSp)\), and
    \(\kcut(\treeRRn)\). But \(\treeSp\) and \(\treeRRn\) are both of logarithmic height. Thus, the
    same method which we used for treating complete binary trees \cite{2018arXiv181105673S} should
    also work.
\end{remark}

\subsubsection{Conditional Galton-Watson trees}\label{sec:gwtrees}

\newcommand{\treeGw}{\mathbb T^{\mathrm{gw}}}
\newcommand{\treeGwn}{\treeGw_{n}}

A Galton-Watson tree \(\treeGw\) is a random tree that starts with the root node and recursively
attaches a random number of children to each node in the tree, where the numbers of children are
drawn independently from the same distribution \(\cL(\xi)\) (the offspring distribution). A
conditional Galton-Watson tree \(\treeGwn\) is \(\treeGw\) restricted to size \(n\).  See
\cite{svante12} for a comprehensive survey of conditional Galton-Watson trees.

\citet[Theorem 1.6]{janson06} showed that \({\kcut_{k}(\treeGwn)}/\sqrt{n}\) converges weakly to a
Rayleigh distribution and the convergence is also in all moments if \(\xi\) has a
finite exponential moment.  In particular
\begin{equation}
    \frac{
        \e{\kcut_{k}(\treeGwn)}
    }{
        \sqrt{n}
    }
    \to
    \E
    {
        \int_{0}^{1} \left( \frac{2 e(t)}{\sigma} \right)^{-1} \mathrm{d} t
    }
    =
    \sigma
    \sqrt{
        \frac{
            \pi
        }{
            2
        }
    }
    ,
    \label{eq:GW:exp:1}
\end{equation}
where \(e(t)\) denotes a normalized Brownian excursion and \(\sigma^{2}=\V{\xi}\).
It is straight forward to adapt the method in
\cite{janson06} to get the first moment of \(\e{\kcut_{r}(\treeGwn)}\). (Though higher moments and the
limit distribution seems to be elusive.) We formulate this as lemma and refer the reader to
\cite{janson06} for details.
\begin{lemma}\label{lem:GW:exp}
    Assume that \(\E{\xi^{3}}<\infty\).
    Then for \(r \in \{1,\dots,k\}\), 
    \begin{equation}
        \frac{
            \e{\kcut_{r}(\treeGwn)}
        }{
            n^{1-\frac{r}{2 k}}
        }
        \to
        \frac{(k!)^{\frac{r}{k}}}{k}\frac{\Gamma(\frac{r}{k})}{\Gamma(r)}
        \E
        {
            \int_{0}^{1} \left( \frac{2 e(t)}{\sigma} \right)^{-\frac{r}{k}} \mathrm{d} t
        }
        =
        \frac{(k!)^{\frac{r}{k}}}{k}\frac{\Gamma(\frac{r}{k})\Gamma\left( 1-\frac{r}{2 k} \right)}{\Gamma(r)}
        \left(
            \frac{ \sigma }{ \sqrt{2} }
        \right)^{\frac{r}{k}}
        .
        \label{eq:GW:exp:r}
    \end{equation}
    As a result, 
    \begin{equation}
        \label{eq:GW:exp:all}
        \e{\kcut(\treeGwn)}
        \sim
        \e{\kcut_{1}(\treeGwn)}
        \sim
        \frac{(k!)^{\frac{1}{k}}}{k}{\Gamma\left(\frac{1}{k}\right)\Gamma\left( 1-\frac{1}{2 k} \right)}
        \left(
            \frac{ \sigma }{ \sqrt{2} }
        \right)^{\frac{1}{k}}
        n^{1-\frac{1}{2 k}}
        .
    \end{equation}
\end{lemma}

\section{Some auxiliary results}\label{sec:int}

\begin{lemma}
    Let \(G_{k}\eql\Gam(k)\).
    Let \(\alpha \eqd \frac{1}{2}\left( \frac{1}{k} + \frac{1}{k+1} \right)\) and \(x_{0} \eqd m^{-\alpha}\).
    Then uniformly for all \(x \in [0,x_0]\),
    \begin{align}
        \p{G_{k} > x}^{m}
        =
        \left( 
            \frac{
                \Gamma\left( k,x \right)
            }{
                \Gamma(k)
            }
        \right)^{m}
        =
        \left( 1+\bigO{m^{-\frac{1}{2k}}} \right)
        \exp\left( -\frac{m x^{k}}{k!} \right),
        \label{eq:gam:approx}
    \end{align}
    where \(\Gamma(\ell,z)\) denotes the upper incomplete gamma function.
    \label{lem:gam:approx}
\end{lemma}

\begin{proof}
    By the density function of gamma distributions,
    \(
        \p{G_{k} > x}
        =
        {\Gamma\left( k, x \right)}/{\Gamma(k)}
        .
    \)
    It then follows from the series expansion of the incomplete gamma function \cite[8.7.3]{NIST}, that
    uniformly for all \(x \le x_{0}\),
    \begin{align}
        \left( 
            \frac{\Gamma\left( k, x \right)}{\Gamma(k)}
        \right)^{m}
        =
        \left( 1 - \frac{x^{k}}{k!} + \bigO{x_{0}^{k+1}} \right)^{m}
        =
        \left( 1+\bigO{m^{-\frac{1}{2k}}} \right)
        \exp\left( -\frac{m x^{k}}{k!} \right),
        \label{eq:mean:i:asympt}
    \end{align}
    where we use that \(-\alpha (k+1)+1=-\frac{1}{2k}\).
\end{proof}

\begin{lemma}
    \label{lem:i:int}
    Let \(G_{k}\eql\Gam(k)\).
    Let \(a \ge 0\) and \(b \ge 1\) be fixed. Then uniformly for \(m \ge 1\),
    \begin{align}
        \int_{0}^{\infty}
        x^{b-1} \erm^{-a x}
        \p{G_{k} > x}^{m}
        \,
        \mathrm{d}x
        =
        \left( 1+\bigO{m^{-\frac{1}{2k}}} \right)
        \frac{(k!)^{\frac{b}{k}} 
        }{k}
        {\Gamma\left(\frac{b}{k}\right)}
        m^{-\frac{b}{k}}
        .
        \label{eq:i:int}
    \end{align}
\end{lemma}

\begin{proof}
    By \autoref{lem:gam:approx}, the left-hand-side of \eqref{eq:i:int} equals
    \begin{align}
        \int_{0}^{x_{0}}
        {x^{b-1} \erm^{-ax}}
        \left( 
            \frac{\Gamma\left( k, x \right)}{\Gamma(k)}
        \right)^{m}
        \,
        \mathrm{d}x
        +
        \int_{x_0}^{\infty}
        {x^{b-1} \erm^{-ax}}
        \left( 
            \frac{\Gamma\left( k, x \right)}{\Gamma(k)}
        \right)^{m}
        \,
        \mathrm{d}x
        \eqd
        A_{1} + A_{2},
        \label{eq:mean:i:A1:A2}
    \end{align}
    where \(x_{0}=m^{-\alpha}\) and \(\alpha=\frac{1}{2}\left( \frac{1}{k}+\frac{1}{k+1} \right)\).
    Then
    \begin{align}
        A_{1}
        &
        =
        \left( 1+\bigO{m^{-\frac{1}{2k}}} \right)
        \int_{0}^{x_{0}}
        {x^{b-1} \erm^{-a x}}
        \exp\left( -\frac{m x^{k}}{k!} \right)
        \,
        \mathrm{d}x
        \\
        &
        =
        \left( 1+\bigO{m^{-\frac{1}{2k}}} \right)
        \frac{(k!)^{\frac{b}{k}} }{k}
        \left(\Gamma\left(\frac{b}{k}\right) - \Gamma\left( \frac{b}{k}, w_{0} \right)  \right)
        m^{-\frac{b}{k}}
        ,
        \label{eq:mean:i:A1}
    \end{align}
    where \(w_{0} = \frac{m x_{0}^{k}}{k!} = \Theta(m^{\frac{1}{2k(k+1)}})\).
    By the upper bound given in \cite[8.11.i]{NIST}, 
    \( \Gamma\left( \frac{b}{k}, w_{0} \right) = \bigO{\erm^{-\frac{w_{0}}{2}}}\), which is exponentially small and
    can be neglected.
    Using \eqref{eq:mean:i:asympt}, one can verify that \(A_{2} = \bigO{\erm^{-\frac{w_{0}}{2}}}\)
    which can also be neglected.
\end{proof}

\begin{lemma}
    \label{lem:int:double:exp}
    For \(a>0\), \(b>0\) and \(k \ge 2\),
    \begin{align}
        \xi_{k}(a,b)
        &
        \eqd
        \int_{0}^{\infty}
        \int_{y}^{\infty}
        \ee^{-a x^k/k!-b y^k/k!}
        \, \mathrm dx
        \, \mathrm dy
        \\
        &
        =
        \frac{\Gamma\left( \frac{2}{k} \right)}{k}
        \left( 
            \frac{k!}{a}
        \right)^{\frac{2}{k}}
        \HGFunc
        \left( \frac{2}{k}, \frac{1}{k}; 1+\frac{1}{k}; -\frac{b}{a} \right)
        \label{eq:doub:exp}
        ,
    \end{align}
    where \(\HGFunc\) denotes the hypergeometric function.
    In particular, 
    \begin{equation}
        \label{eq:doub:exp:2}
        \xi_{2}(a,b) = {\arctan\left(\sqrt{\frac{b}{a}}\right)}{(ab)^{-\frac{1}{2}}}
        .
    \end{equation}
\end{lemma}

\begin{proof}
    Changing to polar system by \(x = r \cos(\theta)\) and \(y = r \sin(\theta)\),
    \begin{align}
        \xi_{k}(a,b)
        =
        &
        \int_{0}^{\pi/4}
        \int_{0}^{\infty}
        \exp
        \left[ 
            -r^{k}
            \left( 
            a \frac{\cos(\theta)^{k}}{k!}
            +
            b \frac{\sin(\theta)^{k}}{k!}
            \right)
        \right]
        r
        \, \mathrm{d}r
        \, \mathrm{d}\theta
        \\
        &
        =
        \int_{0}^{\pi/4}
            \left( 
            a \frac{\cos(\theta)^{k}}{k!}
            +
            b \frac{\sin(\theta)^{k}}{k!}
            \right)^{-\frac{2}{k}}
            \frac{\Gamma(\frac{2}{k})}{k}
        \, \mathrm{d}\theta
        \\
        &
        =
        \frac{\Gamma(\frac{2}{k})}{k}
        \left( \frac{k!}{a} \right)^{\frac{2}{k}}
        \int_{0}^{\pi/4}
            \left( 
            1
            +
            \frac{b}{a} \tan(\theta)^{k}
            \right)^{-\frac{2}{k}}
        \, \mathrm{d}\theta
        \\
        &
        =
        \frac{\Gamma(\frac{2}{k})}{k^2}
        \left( \frac{k!}{a} \right)^{\frac{2}{k}}
        \int_{0}^{1}
            u^{\frac{1}{k}-1}
            \left( 
            1
            +
            \frac{b}{a} u
            \right)^{-\frac{2}{k}}
        \, \mathrm{d}u
        ,
        \label{eq:doub:exp:polar}
    \end{align}
    which equals the right-hand-side of \eqref{eq:doub:exp} by
    \cite[15.6.1]{NIST}.
    For \eqref{eq:doub:exp:2}, see \cite[15.4.3]{NIST}.
\end{proof}

\begin{lemma}
    \label{lem:bound:xi}
    For \(a>0\), \(b>0\) and \(k \ge 2\),
    \begin{align}
        \label{eq:bound:xi:1}
        (a+b)^{-\frac{2}{k}}
        \le
        \frac{
            k
        }{
            \Gamma\left( \frac{2}{k} \right)
            \left( k! \right)^{\frac{2}{k}}
        }
        \xi_{k}(a,b)
        \le
        a^{-\frac{2}{k}}
        +
        b^{-\frac{2}{k}}
        .
    \end{align}
    Moreover, \(\xi_{k}(a,b)\) is monotonically decreasing in both \(a\) and \(b\).
\end{lemma}

\begin{proof}
    Let
    \begin{align}
        \xi_{k}^{*}(a,b)
        \eqd
        \frac{
            k
        }{
            \Gamma\left( \frac{2}{k} \right)
            \left( k! \right)^{\frac{2}{k}}
        }
        \xi_{k}(a,b)
        =
        \left(a+b\right)^{-\frac{2}{k}}
        \HGFunc
        \left( 
            \frac{2}{k},
            1;
            1+\frac{1}{k};
            \frac{b}{a+b}
        \right)
        ,
        \label{eq:bound:xi:5}
    \end{align}
    where we use \cite[15.8.1]{NIST}.
    Let \(\alpha_{1}=\frac{k}{k+1}\). By \citet[cor.~2]{karp15}, for \(x \in (0,1)\),
    \begin{align}
        \left( 1-\alpha_{1} x \right)^{-\frac{2}{k}}
        \le
        \HGFunc
        \left( 
            \frac{2}{k},
            1;
            1+\frac{1}{k};
            x
        \right)
        \le
        1-\alpha_{1}
        +
        \alpha_{1}
        \left( 
            1-x
        \right)^{-\frac{2}{k}}
        \label{eq:bound:xi:3}
        .
    \end{align}
    This together with \eqref{eq:bound:xi:5} give us \eqref{eq:bound:xi:1}.

    For monotonicity, using the derivative formula \cite[15.5.1]{NIST}, it is easy to verify that
    for \(a > 0\) and \(b > 0\)
    \(
        \frac{\partial}{\partial a} 
        \xi_{k}^{*}(a,b)
        <0
    \)
    and
    \(
        \frac{\partial}{\partial b} 
        \xi_{k}^{*}(a,b)
        <0
        .
    \)
\end{proof}

\begin{lemma}
    For \(k \ge 2\), let
    \begin{align}
        \lambda_{k}
        \eqd
        \int_{0}^{1}
        \int_{0}^{1-s}
        \xi_{k}(s,t)
        \, \mathrm{d}t
        \, \mathrm{d}s
        .
        \label{eq:xi:k}
    \end{align}
    Then
    \begin{align}
        \label{eq:xi:k:1}
        \lambda_{k} =
        \left\{
        \begin{array}{*2{>{\displaystyle}l}}
        \frac{\pi  \cot \left(\frac{\pi }{k}\right) \Gamma \left(\frac{2}{k}\right)
            (k!)^{\frac{2}{k}}}{2 \left({k}-2\right) \left({k}-1\right)}
        & 
        k>2,
        \\
        \frac{\pi^2}{4}
        &
        k=2
        .
        \end{array}
        \right.
    \end{align}
    \label{lem:int:hyper}
\end{lemma}

\begin{proof}
    When \(k=2\), applying \eqref{eq:doub:exp:2} and changing to the polar system by letting 
    \(s=(r \cot(\theta))^{2} \) and \(t= (r \sin(\theta))^{2}\), we get
    \begin{align}
        \lambda_{2} 
        =
        \int_{0}^{1}
        \int_{0}^{1-s}
        \frac{\arctan\left(\sqrt{\frac{t}{s}}\right)}{\sqrt{st}}
        =
        \int_{0}^{\frac{\pi}{2}}
        \int_{0}^{1}
        4 r \theta
        \, \mathrm{d}r
        \, \mathrm{d}\theta
        = 
        \frac{\pi^2}{4}
        .
        \label{eq:int:hyper:2}
    \end{align}

    For \(k \ge 3\), by \autoref{lem:int:double:exp}, it suffices to show that
    \begin{align}
        \int_{0}^{1}
            s^{-\frac{2}{k}}
        \int_{0}^{1-s}
            \HGFunc
            \left( \frac{2}{k}, \frac{1}{k}; 1+\frac{1}{k}; -\frac{t}{s} \right)
        \, \mathrm{d}t
        \, \mathrm{d}s
        =
        \frac
        { k \pi \cot\left( \frac{\pi}{k} \right)}
        {2(k-2)(k-1)},
        \label{eq:int:hyper:3}
    \end{align}
    which is easily verifiable using Mathematica. 
    A human proof can be derived using the series expansion of hypergeometric function 
    \cite[15.6.1]{NIST}.
\end{proof}

\begin{remark}
    \label{rm:hyper:cot}
    In an attempt to prove \autoref{lem:int:hyper}, we discovered the following identity
    \begin{align}
        \int_0^{\infty } (w+1)^{\frac{2}{k}-2} \, F\left(\frac{2}{k},\frac{1}{k};1+\frac{1}{k};-w\right) 
        \, 
        \mathrm{d}w
        =
        \frac{\pi  \cot \left(\frac{\pi }{k}\right)}{k-2}
        ,
        &
        &
        \left( k \ge 3 \right)
        ,
        \label{eq:hyper:cot}
    \end{align}
    which we have not found in the literature.
    The proof follows from changing to polar system in the left-hand-side of \eqref{eq:int:hyper:3}
    by letting \(s = (r \cos(\theta))^{k}\) and \(t = (r \sin(\theta))^{k} \).
\end{remark}

\printbibliography{}

\end{document}